\documentclass[twoside,twocolumn,a4paper,11pt]{article}

\usepackage[a4paper,top=0.92cm,bottom=1.75cm,left=0.554cm,right=0.554cm]{geometry}

\usepackage{changepage}

\usepackage{xurl}
\usepackage[hidelinks, 
            pdfauthor={Daniel Vallstrom},
            pdftitle={A logical treatment of noise: Solving The Hardest Logic Puzzle Ever and its generalizations}]
           {hyperref}
\usepackage{graphicx}

\usepackage[numbers, square, comma, compress]{natbib}
\usepackage{amsfonts}
\usepackage{amsthm}
\usepackage{mathtools}
\usepackage{amssymb}
\usepackage{mathrsfs}

\usepackage{siunitx}

\newtheorem{theorem}{Theorem}[section]
\newtheorem{lemma}[theorem]{Lemma}
\theoremstyle{definition}
\newtheorem{defn}[theorem]{Definition}
\newtheorem{solution}[theorem]{Solution}

\usepackage[misc]{ifsym}

\usepackage{titlesec}
\titleformat{\section}      {\normalfont\large\bfseries}     {\thesection}      {0.8em}{}
\titleformat{\subsection}   {\normalfont\normalsize\bfseries}{\thesubsection}   {0.7em}{}
\titleformat{\subsubsection}{\normalfont\normalsize\itshape} {\thesubsubsection}{0.7em}{}
\titlespacing*{\section}      {0pt}{2.88ex plus 1ex minus .2ex}{1.59ex plus .2ex}
\titlespacing*{\subsection}   {0pt}{2.70ex plus 1ex minus .2ex}{1.15ex plus .2ex}
\titlespacing*{\subsubsection}{0pt}{2.70ex plus 1ex minus .2ex}{1.15ex plus .2ex}

\usepackage[font=small]{caption}
\usepackage[T1]{fontenc}

\usepackage{enumitem}
\usepackage{bm}

\title{\vspace{-0.35em}A logical treatment of noise:\\ Solving
  The Hardest
  Logic Puzzle Ever
  and its generalizations}
\author{\hspace{0.25em}Daniel 
Vallstrom\textsuperscript{\href{mailto:daniel.vallstrom@gmail.com}
{\tiny\Letter}}\addtocounter{footnote}{1}\thanks{daniel.vallstrom@gmail.com}}
\date{November 2026}
\begin{document}

\maketitle

\abstract{Raymond Smullyan came up with a puzzle that
George Boolos called The Hardest Logic Puzzle Ever.\cite{bool96}
The puzzle has truthful, lying, and random gods who answer yes or no questions 
with words that we don't know the meaning of.
The challenge is to figure out which type each god is.
The puzzle has attracted some general attention --- 
for example, 
%10 million people have come across the puzzle through \cite{gend17}.
%, there have been 10 million encounters of the puzzle through
one popular presentation of the puzzle has been viewed 10 million times.\cite{gend17}
Various ``top-down'' solutions to the puzzle have been developed.\cite{bool96,rob01}
We present a systematic bottom-up approach
%Here a systematic bottom-up approach 
to the puzzle and its generalization.
% is presented.
We prove that an $n$ gods puzzle is solvable if and only if 
the random gods are less than the non-random gods,
for arbitrary cardinals.
We develop a solution using $4.15$ questions 
on average
to the $5$ gods variant with $2$ random and $3$ lying gods.
%Finally, 
Even though the search space is astronomically large,
we introduce an algorithm and an 
%accompanying 
implementation
%And we derive an algorithm, and an implementation, 
for finding solutions to the generalized problem,
together with upper bounds.
%despite the astronomically large search space.
Finally, we note that 
random gods 
%function as 
act like
noisy sources, 
%connecting <this> 
%providing 
which provides
a connection
to fault-tolerant computing.}
%
%Apropos the solution, there
%There is also an aside on mathematical vs.\ computational thinking.
%}

% p. 1, l
\enlargethispage{0.2\baselineskip}

%\section{\hspace{-10pt}}
%\section{\hspace{-11pt} \hspace{-1pt}}
\section{The Hardest Logic Puzzle Ever}
\begin{defn}[The Hardest Logic Puzzle Ever]
\label{def:3GodsPuzzle}
Three gods ($\gamma_1,\ldots,\gamma_3$) will answer three yes or no questions. 
Each question is to be directed at one god at a time.
The gods answer with the word `$\chi$' (or `\_') but we don't know what `$\chi$' (or `\_') means.
One god ($\mathcal{T}$) always tells the truth, one ($\mathcal{F}$) always lies, 
and one ($\mathcal{R}$) answers 
randomly\footnote{The
\label{fn:R}
puzzle has been interpreted as to allow for the random god to not answer randomly but instead
randomly function as a god who either tells the truth, or lies.\cite{rab08}
%,wikiTHLPE} 
Since this renders the random god pointless,
%\cite{rab08}
%and the whole puzzle somewhat trivial and pointless,
%\cite{rab08} 
we'll stick to the interpretation that the random god answers truly randomly.
% ;
% a straightforward solution to the ``non-random'' interpretation using $2.67$ questions 
% is nonetheless included here as preparation.
%follows from the treatment here.

Besides, the truly random interpretation seems to be 
what 
%how 
Boolos 
had in
%understood the puzzle
mind,\cite{bool96}
e.g.\ with explanations like
``will answer your question yes or no, completely at random''\cite[\hspace{-0.25em}p.\,2]{bool96}.
It is also how \cite{rob01} interpreted the puzzle.}\nolinebreak[3]\hspace{-0.17em}.
The challenge is to figure out which god is which.\cite{bool96}
\end{defn}

\section{Groundwork}
\label{sec:sol}
We'll use $0$, $\bot$, no, and false interchangeably, when there is no risk for confusion.
And similarly for $1$, $\top$, yes, and true.
And $\chi$, whatever it means.
We'll also e.g.\ use '$=$' as a boolean function.

With $3$ gods there are $3\cdot 2$ possibilities for the gods. 
The possibilities doubles if we are to figure out the meaning of $\chi$. 
With $3$ questions we can discern $2^3$ outcomes. 
Hence, we better remain 
%agnostic
ignorant
of the meaning of $\chi$
if we are allowed only $3$ questions.

% p. 1, r
\enlargethispage{0.2\baselineskip}

\subsection{A question template}
Let $\gamma(q)$ be 
%the 
god $\gamma$'s answer to the question $q$.

Given a yes or no question $q$,
and a god $\gamma$,
we want a function $t(q,\gamma,\chi) \rightarrow\left\lbrace 0,1\right\rbrace$ that gives 
the truth value of $q$
when $\gamma$ isn't the random god:
%
% Skip below sentence, and move colon up, to 'god:'. vvv???
%Fortunately, one of the first and simplest attempts at such a $t$ yields one that works:  
\begin{defn}
\label{def:t}
\(
t(q,\gamma)\coloneqq\;\;
%=_{\mbox{\footnotesize def}}\;\; 
 \gamma( \mbox{``}\gamma(q)=\chi\mbox{''}) = \chi 
\)
\end{defn}
\begin{theorem}
\label{theorem:templateWorks}
If $\gamma\neq\mathcal{R}$, then $t(q,\gamma)\leftrightarrow q$ 
%is the truth value for $q$.
\end{theorem}
\vspace{-0.95em}
\begin{proof}
We'll go through all possible cases:

\begin{description}[itemsep=-3pt,topsep=1pt]
%\vspace{-0.55em}
%\setlength{\itemsep}{-5pt}
\item[\(\bm{q\!=\!1, \gamma\!=\!\mathcal{T}, \chi\!=\!1}\):] \hspace{-0.0em}Then 
$\gamma(q)\!=\!\chi$, and
\( \gamma( \mbox{``}\gamma(q)\!=\!\chi\mbox{''}) \!=\! \chi \).
\item[\(\bm{q\!=\!1, \gamma\!=\!\mathcal{T}, \chi\!=\!0}\):] \hspace{-0.0em}Then 
$\gamma(q)\!\neq\!\chi$, and
\( \gamma( \mbox{``}\gamma(q)\!=\!\chi\mbox{''}) \!=\! \chi \).
\item[\(\bm{q\!=\!1, \gamma\!=\!\mathcal{F}, \chi\!=\!1}\):] \hspace{-0.0em}Then 
$\gamma(q)\!\neq\!\chi$, and
\( \gamma( \mbox{``}\gamma(q)\!=\!\chi\mbox{''}) \!=\! \chi \).
\item[\(\bm{q\!=\!1, \gamma\!=\!\mathcal{F}, \chi\!=\!0}\):] \hspace{-0.0em}Then 
$\gamma(q)\!=\!\chi$, and
\( \gamma( \mbox{``}\gamma(q)\!=\!\chi\mbox{''}) \!=\! \chi \).
\item[\(\bm{q\!=\!0, \gamma\!=\!\mathcal{T}, \chi\!=\!1}\):] \hspace{-0.0em}Then 
$\gamma(q)\!\neq\!\chi$, and
\( \gamma( \mbox{``}\gamma(q)\!=\!\chi\mbox{''}) \!\neq\! \chi \).
\item[\(\bm{q\!=\!0, \gamma\!=\!\mathcal{T}, \chi\!=\!0}\):] \hspace{-0.0em}Then 
$\gamma(q)\!=\!\chi$, and
\( \gamma( \mbox{``}\gamma(q)\!=\!\chi\mbox{''}) \!\neq\! \chi \).
\item[\(\bm{q\!=\!0, \gamma\!=\!\mathcal{F}, \chi\!=\!1}\):] \hspace{-0.0em}Then 
$\gamma(q)\!=\!\chi$, and
\( \gamma( \mbox{``}\gamma(q)\!=\!\chi\mbox{''}) \!\neq\! \chi \).
\item[\(\bm{q\!=\!0, \gamma\!=\!\mathcal{F}, \chi\!=\!0}\):] \hspace{-0.0em}Then 
$\gamma(q)\!\neq\!\chi$, and
\( \gamma( \mbox{``}\gamma(q)\!=\!\chi\mbox{''}) \!\neq\! \chi \).
\end{description}
%\vspace{-0.35em}
(Cases also hold for symmetry reasons, and for double negation reasons.)
\end{proof}

For convenience we'll also introduce a way to refer to the meta-question put to the god 
in definition \ref{def:t}:
\begin{defn}
\label{def:tq}
\(
t_q(q,\gamma)\coloneqq\;\;
%=_{\mbox{\footnotesize def}}\;\; 
 \mbox{``}\gamma(q)=\chi\mbox{''} 
\)
\end{defn}

Boolos does not use something like definition \ref{def:tq} in his solution to the 
puzzle,\cite{bool96}
but Tim Roberts does.\cite{rob01}
%,rab08} do.

If $\gamma\neq\mathcal{R}$, then since $t(q,\gamma)\leftrightarrow q$, 
any number of nested $t_q$ applications would work too
(cf.\ section \ref{sec:fixNote}):
\[
 \gamma( \mbox{``}\gamma(\mbox{``}\ldots \gamma( \mbox{``}\gamma(q)=\chi\mbox{''}) =
  \chi\ldots\mbox{''}) = \chi\mbox{''}) = \chi 
 \;\leftrightarrow\;  q
\]
% be equivalent with $q$ too, and will work as 
%(Cf.\ section \ref{sec:fixNote}.)

\subsection{How to find questions}
Instead of presenting ``top-down'' solutions,
% that ``magically'' work, 
we'll try to develop solutions from the ground up,
that are guaranteed to work, and are optimal.

Essential for efficient searches is to split the search space in equally large subparts.
This means that we want the possible answers to a question to be equally strong.
%, and then repeat the process.

\subsubsection{Finding a solution to the non-random interpretation}
Let's suppose, 
as preparation,
%for pedagogical purposes,
that $\mathcal{R}$ functions like either 
$\mathcal{T}$ or $\mathcal{F}$ (cf.\ fn.\,\ref{fn:R});
in particular, theorem \ref{theorem:templateWorks} works also for $\mathcal{R}$.

%\enlargethispage{0.35\baselineskip}

Then
%, disregarding any randomness from $\mathcal{R}$ for the time being,
an optimal split of the $6$ possibilities would be provided 
%with a question like this:
by asking about:
\begin{defn}[$q_{\bar{\mathcal{R}}}$]
\label{def:qRBar}
\begin{equation}
\label{eq:qRBar}
q_{\bar{\mathcal{R}}}\coloneqq\;\;
\bigvee
 \left\lbrace
  \begin{matrix}
   \land\{\gamma_1=\mathcal{R},\gamma_2=\mathcal{T},\gamma_3=\mathcal{F}\},\\
   \land\{\gamma_1=\mathcal{R},\gamma_2=\mathcal{F},\gamma_3=\mathcal{T}\},\\
   \land\{\gamma_1=\mathcal{T},\gamma_2=\mathcal{F},\gamma_3=\mathcal{R}\}\;\\
  \end{matrix}
 \right\rbrace 
\end{equation}
\end{defn}
We'll also need to reason about $\neg q_{\bar{\mathcal{R}}}$.
\begin{equation}
\neg q_{\bar{\mathcal{R}}} \leftrightarrow
\bigwedge
 \left\lbrace
  \begin{matrix}
   \lor\{\gamma_1\neq\mathcal{R},\gamma_2\neq\mathcal{T},\gamma_3\neq\mathcal{F}\},\\
   \lor\{\gamma_1\neq\mathcal{R},\gamma_2\neq\mathcal{F},\gamma_3\neq\mathcal{T}\},\\
   \lor\{\gamma_1\neq\mathcal{T},\gamma_2\neq\mathcal{F},\gamma_3\neq\mathcal{R}\}\;\\
  \end{matrix}
 \right\rbrace 
\end{equation}
which in disjunctive normal form (DNF) becomes
\begin{equation}
\label{eq:notqRBar}
\neg q_{\bar{\mathcal{R}}} \leftrightarrow
\bigvee
 \left\lbrace
  \begin{matrix}
   \land\{\gamma_1=\mathcal{T},\gamma_2=\mathcal{R},\gamma_3=\mathcal{F}\},\\
   \land\{\gamma_1=\mathcal{F},\gamma_2=\mathcal{T},\gamma_3=\mathcal{R}\},\\
   \land\{\gamma_1=\mathcal{F},\gamma_2=\mathcal{R},\gamma_3=\mathcal{T}\}\;\\
  \end{matrix}
 \right\rbrace 
\end{equation}
\begin{solution}
\label{solution:nonR}
There is a solution to the non-random interpretation of The Hardest Logic Puzzle Ever
using $2.67$ questions.
(It's also enough to assume non-randomness for only the first question.)
\end{solution}
\begin{proof}
Put $t_q(q_{\bar{\mathcal{R}}},\gamma_1)$ to $\gamma_1$ and consider the possible cases:

\begin{description}[itemsep=2pt,topsep=1pt,leftmargin=1.1em]
%\vspace{-0.55em}
%\setlength{\itemsep}{0pt}
\item[Case \(t(q_{\bar{\mathcal{R}}},\gamma_1)\):] \hspace{0.0em}Then we know
 from theorem \ref{theorem:templateWorks}, and from the supposed non-randomness of $\mathcal{R}$,
 that $q_{\bar{\mathcal{R}}}$ holds.
 Next we'll ask $\gamma_2$ about $\gamma_2=\mathcal{T}$.
 
 \textbf{Suppose $t(\gamma_2=\mathcal{T},\gamma_2)$.}
 
 Then we know from equation \eqref{eq:qRBar} that
 $\gamma_1=\mathcal{R}$, $\gamma_2=\mathcal{T}$, and $\gamma_3=\mathcal{F}$.
 
 This case used $2$ questions.
 
 \textbf{Suppose $\neg t(\gamma_2=\mathcal{T},\gamma_2)$.}
 
 Then we know from equation \eqref{eq:qRBar} that $\gamma_2=F$.
 
 $t(\gamma_1=\mathcal{T},\gamma_2)$ determines which is which of $\gamma_1$ and $\gamma_3$.

\item[Case \(\neg t(q_{\bar{\mathcal{R}}},\gamma_1)\):] \hspace{0.0em}Then 
 $\neg q_{\bar{\mathcal{R}}}$.
 Next we'll ask $\gamma_1$ about $\gamma_1=\mathcal{T}$.
 
 \textbf{Suppose $t(\gamma_1=\mathcal{T},\gamma_1)$.}
 
 Then we know from equation \eqref{eq:notqRBar} that
 $\gamma_1=\mathcal{T}$, $\gamma_2=\mathcal{R}$, and $\gamma_3=\mathcal{F}$.
 
 This case used $2$ questions.
 
 \textbf{Suppose $\neg t(\gamma_1=\mathcal{T},\gamma_1)$.}
 
 Then we know from equation \eqref{eq:notqRBar} that $\gamma_1=F$.
 
 $t(\gamma_2=\mathcal{T},\gamma_1)$ determines which is which of $\gamma_2$ and $\gamma_3$.
\end{description}
\vspace{-0.65em}
\end{proof}
\vspace{-0.90em}

\subsubsection{Managing randomness}
%Handling, Coping with

For the full
%, random,
version of the puzzle (def.\,\ref{def:3GodsPuzzle}),
each question that can be answered by the random god weakens the 
narrowing of the search space by up to $2$ possibilities,
%e.g.\ 
adding $(\gamma_1\!=\!\mathcal{R},\gamma_2\!=\!\mathcal{F},\gamma_3\!=\!\mathcal{T})$
and $(\gamma_1\!=\!\mathcal{R},\gamma_2\!=\!\mathcal{T},\gamma_3\!=\!\mathcal{F})$ 
to the list of possibilities, say.
So it also seems important to find, as quickly as possible,
a god that isn't $\mathcal{R}$, in order to get reliable answers and minimize waste.

%\enlargethispage{0.35\baselineskip}

Hence, the problem with asking about something like $q_{\bar{\mathcal{R}}}$ for the full
%, random, 
puzzle (def.\,\ref{def:3GodsPuzzle})
is that the conclusion we are able to draw from $\neg t(q_{\bar{\mathcal{R}}},\gamma_1)$
is given by
\[
\bigvee
 \left\lbrace
  \begin{matrix}
   \land\{\gamma_1=\mathcal{T},\gamma_2=\mathcal{R},\gamma_3=\mathcal{F}\},\\
   \land\{\gamma_1=\mathcal{F},\gamma_2=\mathcal{T},\gamma_3=\mathcal{R}\},\\
   \land\{\gamma_1=\mathcal{F},\gamma_2=\mathcal{R},\gamma_3=\mathcal{T}\},\\
   \land\{\gamma_1=\mathcal{R},\gamma_2=\mathcal{F},\gamma_3=\mathcal{T}\},\\
   \land\{\gamma_1=\mathcal{R},\gamma_2=\mathcal{T},\gamma_3=\mathcal{F}\}\;\\
  \end{matrix}
 \right\rbrace 
\]
Since this has $5$ possibilities, it's not solvable with only the remaining $2$ questions.
(The case $t(q_{\bar{\mathcal{R}}},\gamma_1)$ remains as in solution \ref{solution:nonR} though
as that case already includes the possibilities
$(\mathcal{R},\mathcal{F},\mathcal{T})$ and $(\mathcal{R},\mathcal{T},\mathcal{F})$.)

Instead we'll again balance the partitions, by moving $1$ of the added random possibilities from
the negative side to the positive:
\begin{defn}[$q_1$]
\label{def:q1}
\begin{equation}
\label{eq:q1}
q_1\coloneqq\;\;
\bigvee
 \left\lbrace
  \begin{matrix}
   \land\{\gamma_1=\mathcal{R},\gamma_2=\mathcal{T},\gamma_3=\mathcal{F}\},\\
   \land\{\gamma_1=\mathcal{R},\gamma_2=\mathcal{F},\gamma_3=\mathcal{T}\},\\
   \land\{\gamma_1=\mathcal{T},\gamma_2=\mathcal{F},\gamma_3=\mathcal{R}\},\\
   \land\{\gamma_1=\mathcal{F},\gamma_2=\mathcal{T},\gamma_3=\mathcal{R}\}\;\\
  \end{matrix}
 \right\rbrace 
\end{equation}
\end{defn}
We'll also need to reason about $\neg q_1$.
\begin{equation}
\neg q_1 \leftrightarrow
\bigwedge
 \left\lbrace
  \begin{matrix}
   \lor\{\gamma_1\neq\mathcal{R},\gamma_2\neq\mathcal{T},\gamma_3\neq\mathcal{F}\},\\
   \lor\{\gamma_1\neq\mathcal{R},\gamma_2\neq\mathcal{F},\gamma_3\neq\mathcal{T}\},\\
   \lor\{\gamma_1\neq\mathcal{T},\gamma_2\neq\mathcal{F},\gamma_3\neq\mathcal{R}\},\\
   \lor\{\gamma_1\neq\mathcal{F},\gamma_2\neq\mathcal{T},\gamma_3\neq\mathcal{R}\}\;\\
  \end{matrix}
 \right\rbrace 
\end{equation}
which in disjunctive normal form becomes
\begin{equation}
\label{eq:notq1}
\neg q_1 \leftrightarrow
\bigvee
 \left\lbrace
  \begin{matrix}
   \land\{\gamma_1=\mathcal{T},\gamma_2=\mathcal{R},\gamma_3=\mathcal{F}\},\\
   \land\{\gamma_1=\mathcal{F},\gamma_2=\mathcal{R},\gamma_3=\mathcal{T}\}\;\\
  \end{matrix}
 \right\rbrace 
\end{equation}
which with added $\mathcal{R}$ possibilities gives
\begin{defn}[$\bar{q}_1^R$]
\label{def:notq1R}
\begin{equation}
\label{eq:notq1R}
\bar{q}_1^R \coloneqq\;\;
\bigvee
 \left\lbrace
  \begin{matrix}
   \land\{\gamma_1=\mathcal{T},\gamma_2=\mathcal{R},\gamma_3=\mathcal{F}\},\\
   \land\{\gamma_1=\mathcal{F},\gamma_2=\mathcal{R},\gamma_3=\mathcal{T}\},\\
   \land\{\gamma_1=\mathcal{R},\gamma_2=\mathcal{F},\gamma_3=\mathcal{T}\},\\
   \land\{\gamma_1=\mathcal{R},\gamma_2=\mathcal{T},\gamma_3=\mathcal{F}\}\;\\
  \end{matrix}
 \right\rbrace 
\end{equation}
\end{defn}

\section{A bottom-up solution to the puzzle}
\begin{solution}
\label{solution:bottomUp}
A solution to The Hardest Logic Puzzle Ever exists.
\end{solution}
\begin{proof}
Put $t_q(q_1,\gamma_1)$ to $\gamma_1$ and consider the possible cases:
\begin{description}[itemsep=2pt,topsep=3pt,leftmargin=1.1em]
%\vspace{-0.55em}
%\setlength{\itemsep}{0pt}
\item[Case \(t(q_1,\gamma_1)\):] \hspace{0.2em}Since equation \eqref{eq:q1}
 already includes all the possibilities where $\gamma_1=\mathcal{R}$,
 $q_1$ holds.
 Hence $\gamma_2\neq\mathcal{R}$ and is safe to question.

 Next we'll ask $\gamma_2$ about $\gamma_2=\mathcal{T}$.
 
 After that, $t(\gamma_1=\mathcal{R},\gamma_2)$ determines which is which of $\gamma_1$ and
 $\gamma_3$.

\item[Case \(\neg t(q_1,\gamma_1)\):] \hspace{0.2em}Then $\gamma_1=\mathcal{R}$ or
 $\neg q_1$ holds. Adding the $\gamma_1=\mathcal{R}$ possibilities to $\neg q_1$ results in
 $\bar{q}_1^R$, which must hold. Inspecting equation \eqref{eq:notq1R} shows that
 $\gamma_3\neq\mathcal{R}$ and is safe to question.

 Next we'll ask $\gamma_3$ about $\gamma_3=\mathcal{T}$.
 
 After that, $t(\gamma_1=\mathcal{R},\gamma_3)$ determines which is which of $\gamma_1$ and
 $\gamma_2$.
\end{description}
\vspace{-1.5em}
\end{proof}
%\vspace{-0.15em}

Note that $q_1\leftrightarrow (\gamma_2\!\neq\!\mathcal{R})$. A first question $q$ such that 
$q$ and $\neg q$
are equally strong, and where equally many $\mathcal{R}$ possibilities are added to each side
of the search split, works too --- for example 
$\gamma_3\!=\!\mathcal{R} \lor (\gamma_1\!=\!\mathcal{R} \land \gamma_2\!=\!\mathcal{T} \land
 \gamma_3\!=\!\mathcal{F})$.

Note too that there is no solution using less than $3$ questions to the full puzzle.
%, because $2$ possibilities, stemming from $\gamma_1\!=\!\mathcal{R}$, are added after 
%the first question, while there are
%$3\cdot 2$ possibilities for the gods. 
%With $3$ questions we can discern $2^3$ outcomes
%$3\cdot 2 + 2 = 2^3$

% The bottom-up solution also works for the generalized n-god problem, where one god is 
% $\mathcal{R}$ and 
% the rest are different versions of $\mathcal{T}$ or $\mathcal{F}$. 

\section{The n gods puzzle class}
\label{sec:nGods}
\begin{defn}[The Hardest Logic Puzzle Ever with $n$ Gods, $m$ Random Gods, and
$k$ Truthful Gods]
\label{def:nGodsPuzzle}
Let the $(n,m,k)$ gods puzzle be like The Hardest Puzzle Ever (def.\,\ref{def:3GodsPuzzle})
but with $n$ gods, $m$ random gods, $k$ truthful gods, 
$n-m-k$ lying gods,
and with no restriction on the number of questions allowed.
\end{defn}

\begin{theorem}
\label{theorem:nGodsSolvability}
An $n$ gods puzzle is solvable if and only if the number of random gods is strictly less than
the number of non-random gods.
\end{theorem}

\begin{lemma}
\label{lemma:findingNonR}
If an $n$ gods puzzle has strictly more non-random gods than random gods, 
then a non-random god can be found.  
\end{lemma}
\vspace{-1.2em}
\begin{proof}
We'll prove the lemma by induction.

The lemma holds for puzzles with $1$ and $2$ gods.

Assume that the lemma holds for $k<n$,
and that there are more non-random gods than random gods. 
We'll then find a non-random god for the $n$ case.

Ask $\gamma_1$ about $\gamma_i\!=\!\mathcal{R}$ for $2\leq i\leq n$ until 
$t(\gamma_i\!=\!\mathcal{R},\gamma_1)$, or 
%$i=n$.
all gods have been checked.

If $\gamma_1\!\neq\! \mathcal{R}$, then $\gamma_i\!=\!\mathcal{R}$ by 
theorem \ref{theorem:templateWorks}, or there are no
more random gods at all.

Hence, the subproblem for 
$(\gamma_2,\ldots,\gamma_{i-1},\gamma_{i+1},\linebreak[1]\ldots,\linebreak[1]\gamma_n)$
has at least $1$ less random god, and at most $1$ less non-random god, 
if there is any random god at all. 
Thus, a non-random god $\gamma_j$ can be found for the subproblem. 
And $\gamma_j$ 
%works
%is what we are seeking 
suffice
for the $n$ case too. 
\end{proof}

\begin{lemma}
\label{lemma:RsLTNonRsSolvable}
If an $n$ gods puzzle has strictly more non-random gods than random gods, 
then it is solvable.  
\end{lemma}
\vspace{-1.2em}
\begin{proof}
Assume more non-random gods than random gods. Then, by lemma \ref{lemma:findingNonR},
a non-random god $\gamma_j$ can be found.
After that it's straightforward to go through all gods and ask $\gamma_j$ about their
identity. This determines all the gods, according to theorem \ref{theorem:templateWorks}.
%\vspace{-0.1em}
\end{proof}

\begin{proof}[Proof of theorem \ref{theorem:nGodsSolvability}]
The $\Leftarrow$ case is covered by lemma \ref{lemma:RsLTNonRsSolvable}.

For the $\Rightarrow$ case,
to see why puzzles with at least as many random gods 
as non-random gods aren't solvable, consider the easiest to solve of such problems:
Assume that $n$ is even, that the random gods equal the non-random gods,
and that there are only truthful non-random gods.
Let 
\begin{defn}
\label{def:pnUns}
$p_n \coloneqq $
 \begin{equation}
 \label{eq:nonSolvable}
  \bigvee\!\!
   \left\lbrace
    \begin{matrix}
     \land\{\gamma_1\!=\!\mathcal{T},\gamma_2\!=\!\mathcal{R},\gamma_3\!=\!\mathcal{T},
     \gamma_4\!=\!\mathcal{R},\ldots,\gamma_n\!=\!\mathcal{R}\},\\
     \land\{\gamma_1\!=\!\mathcal{R},\gamma_2\!=\!\mathcal{T},\gamma_3\!=\!\mathcal{R},
     \gamma_4\!=\!\mathcal{T},\ldots,\gamma_n\!=\!\mathcal{T}\}\;\\
    \end{matrix}\!
   \right\rbrace\!\!
 \end{equation}
 % This eq. is enough; we don't have to add that randoms are equal or more than
 % non-randoms. E.g.\ n=3 is fine, as long as both conjunctions can hold.
\end{defn}
Let $p_{\mathcal{\scriptscriptstyle T}\mathcal{\scriptscriptstyle R}}$ and
$p_{\mathcal{\scriptscriptstyle R}\mathcal{\scriptscriptstyle T}}$ be the
conjunctions of $p_n$:
\begin{defn}
\label{pTR}
$p_{\mathcal{\scriptscriptstyle T}\mathcal{\scriptscriptstyle R}}\coloneqq $
 \begin{equation}
   \land\{\gamma_1\!=\!\mathcal{T},\gamma_2\!=\!\mathcal{R},\gamma_3\!=\!\mathcal{T},
   \gamma_4\!=\!\mathcal{R},\ldots,\gamma_n\!=\!\mathcal{R}\}
 \end{equation}
\end{defn}
\begin{defn}
\label{pRT}
$p_{\mathcal{\scriptscriptstyle R}\mathcal{\scriptscriptstyle T}}\coloneqq $
 \begin{equation}
  \hspace{3.0em}\land\{\gamma_1\!=\!\mathcal{R},\gamma_2\!=\!\mathcal{T},
  \gamma_3\!=\!\mathcal{R},
  \gamma_4\!=\!\mathcal{T},\ldots,\gamma_n\!=\!\mathcal{T}\}
 \end{equation}
\end{defn}
If the gods are as described by 
$p_{\mathcal{\scriptscriptstyle T}\mathcal{\scriptscriptstyle R}}$ or
$p_{\mathcal{\scriptscriptstyle R}\mathcal{\scriptscriptstyle T}}$, and only
$p_n$ is known, then that puzzle instance is unsolvable, if the random gods are 
unhelpful (we'll show this next).
So let's assume that the gods are as described by $p_n$.

To get to an unsolvable position, we'll have the random gods ``happen''
to force the puzzle there. To that end,
%given already known conclusions $Q$,
if both $q$ and $\neg q$ are consistent with $p_n$,
(i.e.\ not $p_n\!\rightarrow\!\neg q$ and not $p_n\!\rightarrow\!q$),
assume that the random gods always happen to give the incorrect answer
to a question about $q$.
If only one of $q$ and $\neg q$ is consistent with $p_n$,
assume that the random gods always happen to give the correct answer 
to a question about $q$.

%  ---
% on questions that doesn't follow yet from what is known that is. 
% For already known answers they should happen to answer non-contradictory. 
% For example, they shouldn't say that $\bot$ is true.

Call a question `trivial' if it or its negation follow from already
asked questions.

Given this setup,
and if non-trivial questions are asked,
then it will be known that $p_n$ holds.
But once there,
no more conclusion can be drawn from the answer to a question $q$.
And it is not possible to determine which of 
$p_{\mathcal{\scriptscriptstyle T}\mathcal{\scriptscriptstyle R}}$ and
$p_{\mathcal{\scriptscriptstyle R}\mathcal{\scriptscriptstyle T}}$
that holds.

More specifically, without loss of generality we will assume that $\gamma_1$ is asked about $q$. 

% To see that $p_n$ will be concluded if non-trivial questions are asked,
% note that what we are able to conclude are
% $q \vee \gamma_1\!=\!\mathcal{R}$ from $t(q,\gamma_1)$, and
% $\neg q \vee \gamma_1\!=\!\mathcal{R}$ from $\neg t(q,\gamma_1)$. 
% Since $\gamma_1\!=\!\mathcal{R}$ is consistent with $p_n$, any disjunction with it is too.
% (Alternatively, 
% if $\gamma_1\!=\!\mathcal{T}$, then an answer that $q$ (or $\neg q$) holds
% implies that $q$ (or $\neg q$) must be consistent with $p_n$ since $p_n$ is in fact true,
% and $\gamma_1$ is non-random. 
% %Besides, $\gamma_1\!=\!\mathcal{R}$ too is consistent with $p_n$.
% And if $\gamma_1\!=\!\mathcal{R}$, then again $q$ (or $\neg q$) is consistent with $p_n$  
% because of the assumptions of how random gods answer.)
% %1) q,-q both consi
% %2) only one consi

%Eventually $p_n$ will be concluded from any ordinary non-trivial questions. However, 
%the easiest way to 
An easy way to see that we 
%can
are able to
conclude $p_n$ is to note that we can ask
%If nothing else, eventually
$\frac{n}{2}+1$ gods
% are asked 
about $p_n$ explicitly.
Since all gods answer that $p_n$ holds, $p_n$ can be concluded since at least
$1$ of $\frac{n}{2}+1$ gods must be non-random.

% we are able to conclude only
% $q \vee \gamma_1\!=\!\mathcal{R}$ (or $\neg q \vee \gamma_1\!=\!\mathcal{R}$).
% And the false answer that $q$ (or $\neg q$) holds
% undoes the disjunct $\gamma_1=\mathcal{R}$. 

Note that
if $p_n$ 
%holds
is known
and $q$ is non-trivial, then
$q\leftrightarrow p_{\mathcal{\scriptscriptstyle T}\mathcal{\scriptscriptstyle R}}$ or
$q\leftrightarrow p_{\mathcal{\scriptscriptstyle R}\mathcal{\scriptscriptstyle T}}$.
Because suppose 
that $q$ is non-trivial. Then it has a model $\mathcal{M}$ (i.e.\ an assignment of the gods)
where $q$, and $p_n$, are true. Suppose, without loss of generality, that it's 
the $p_{\mathcal{\scriptscriptstyle T}\mathcal{\scriptscriptstyle R}}$
disjunct that's true in $\mathcal{M}$.
Then, since $p_{\mathcal{\scriptscriptstyle T}\mathcal{\scriptscriptstyle R}}$
completely determines a model, $q$ holds whenever 
$p_{\mathcal{\scriptscriptstyle T}\mathcal{\scriptscriptstyle R}}$ does, i.e.\
$p_{\mathcal{\scriptscriptstyle T}\mathcal{\scriptscriptstyle R}}\rightarrow q$.
Suppose $\neg p_{\mathcal{\scriptscriptstyle T}\mathcal{\scriptscriptstyle R}}$.
Then $p_{\mathcal{\scriptscriptstyle R}\mathcal{\scriptscriptstyle T}}$.
Since $p_{\mathcal{\scriptscriptstyle R}\mathcal{\scriptscriptstyle T}}$
determines its models completely, if $q$ where to hold, then $q$ would
hold in all models to $p_n$, contradicting that $q$ is non-trivial.
Hence $\neg q$ holds, i.e.\ 
$\neg p_{\mathcal{\scriptscriptstyle T}\mathcal{\scriptscriptstyle R}}\rightarrow \neg q$.

% otherwise. Then there are models 
% (i.e.\ assignments of the gods)
% such that 
% $q \wedge p_{\mathcal{\scriptscriptstyle T}\mathcal{\scriptscriptstyle R}}$ is true in one and
% $q \wedge p_{\mathcal{\scriptscriptstyle R}\mathcal{\scriptscriptstyle T}}$ is true in an other.
% Since $p_n$ holds, one of its disjuncts must hold.
% Since 
% $p_{\mathcal{\scriptscriptstyle T}\mathcal{\scriptscriptstyle R}}$ and
% $p_{\mathcal{\scriptscriptstyle R}\mathcal{\scriptscriptstyle T}}$
% completely determine a model, 
% $p_{\mathcal{\scriptscriptstyle T}\mathcal{\scriptscriptstyle R}}$ and
% $p_{\mathcal{\scriptscriptstyle R}\mathcal{\scriptscriptstyle T}}$
% desribe all of the models of $p_n$. Hence $q$ is true in all models of $p_n$,
% contradicting the assumption that $q$ is non-trivial.

Assume that $p_n$ is known.
To see that once $p_n$ is known, no more conclusions can be made, 
suppose $\gamma_1\!=\!\mathcal{T}$, 
and $q$ is non-trivial.
Then an answer that $q$ (or $\neg q$) holds 
(with implications that 
$p_{\mathcal{\scriptscriptstyle T}\mathcal{\scriptscriptstyle R}}$ holds)
is undone because we can only conclude that
$q \vee \gamma_1\!=\!\mathcal{R}$ (or $\neg q \vee \gamma_1\!=\!\mathcal{R}$),
which is equivalent to $p_n$.

Similarly, if instead $\gamma_1=\mathcal{R}$, then we are able to conclude only
$q \vee \gamma_1\!=\!\mathcal{R}$ (or $\neg q \vee \gamma_1\!=\!\mathcal{R}$).
If $q$ is non-trivial, 
then both $q$ and $\neg q$ are consistent with $p_n$.
Hence, the false answer that $q$ (or $\neg q$) holds
undoes the disjunct $\gamma_1=\mathcal{R}$. 

If $q$ is trivial, the concluded disjunction, 
$q \vee \gamma_1\!=\!\mathcal{R}$ (or $\neg q \vee \gamma_1\!=\!\mathcal{R}$),
will already be known, and nothing new can be concluded again.

% More formally, if $q$ is non-trivial, then both $q$ and $\neg q$ are consistent. But in their
% models, due to eq.\,\eqref{eq:nonSolvable}, everything is determined. Hence, no more conclusions
% can be made from any answer.
% This works for when eq. 8 isn't known yet to solver. g1=T case works because g1=R is always
% added as disjunct.
% g1=R works because that's how R answers, we'll assume. R lies only w.r.t. eq 8, not what's 
% known to solver. if q isn't consistent with eq 8, don't answer/lie that it holds.
% So situation will move towards eq. 8, and then stay there.

To see that even harder puzzles to solve are unsolvable too,
the restriction that the non-random gods are only truthful is immaterial.

If more random gods are added, the proof still works with minor alterations, e.g.\
to eq.\,\eqref{eq:nonSolvable}. 
% If there are more random than non-random gods, the problem can be reduced to
% a problem with equally many random as non-random gods.
% (Alternatively, it ought to be possible to reduce the
% problem to one where the random and non-random gods are equally many.)
\end{proof}

\section{A solution to the 5 gods puzzle with 2 random and 3 truthful gods}
\label{sec:t3r2}

We will solve the puzzle where the non-random gods are the same,
i.e.\ all $\mathcal{T}$, or all $\mathcal{F}$. 
This is unimportant
although it reduces the number of possibilities; 
the other variants are similar.

There are $\frac{5\cdot 4}{2}$ possibilities for the gods.

We'll address $\gamma_1$ first, without loss of generality.

Included in any conclusions drawn from the first answer is that $\gamma_1$ could be $\mathcal{R}$:
$\gamma_1\!=\!\mathcal{R} \;\leftrightarrow$
%$\gamma_1\!=\!\mathcal{R}\leftrightarrow$
\begin{equation}
\label{eq:gamma1IsR}
\bigvee\!\!
 \left\lbrace\!
  \begin{matrix}
   \land\{\gamma_1=\mathcal{R},\gamma_2=\mathcal{T},\gamma_3=\mathcal{T},\gamma_4=\mathcal{T},
   \gamma_5=\mathcal{R}\},\\
   \land\{\gamma_1=\mathcal{R},\gamma_2=\mathcal{T},\gamma_3=\mathcal{T},\gamma_4=\mathcal{R},
   \gamma_5=\mathcal{T}\},\vspace{0pt}\\
   \land\{\gamma_1=\mathcal{R},\gamma_2=\mathcal{T},\gamma_3=\mathcal{R},\gamma_4=\mathcal{T},
   \gamma_5=\mathcal{T}\},\\
   \land\{\gamma_1=\mathcal{R},\gamma_2=\mathcal{R},\gamma_3=\mathcal{T},\gamma_4=\mathcal{T},
   \gamma_5=\mathcal{T}\}\;\\
  \end{matrix}\!
 \right\rbrace\!\!\!\!
\end{equation}
For the first question we'll take half of the conjunctions from
$\gamma_1=\mathcal{R}$ (eq.\,\eqref{eq:gamma1IsR}), 
and add half of the remaining possibilities,
aiming to get $\gamma_2$ likely to be non-random in the positive case,
and $\gamma_3$ likely non-random in the negative case:
\begin{defn}%[$q_1^5$] 
\label{def:q15}
$q_1^5\coloneqq$
\begin{equation}
\label{eq:q15}
\bigvee\!
 \left\lbrace\!
  \begin{matrix}
   \land\{\gamma_1=\mathcal{R},\gamma_2=\mathcal{T},\gamma_3=\mathcal{T},\gamma_4=\mathcal{T},
   \gamma_5=\mathcal{R}\},\\
   \land\{\gamma_1=\mathcal{R},\gamma_2=\mathcal{T},\gamma_3=\mathcal{T},\gamma_4=\mathcal{R},
   \gamma_5=\mathcal{T}\},\\
   \land\{\gamma_1=\mathcal{T},\gamma_2=\mathcal{T},\gamma_3=\mathcal{T},\gamma_4=\mathcal{R},
   \gamma_5=\mathcal{R}\},\vspace{0pt}\\
   \land\{\gamma_1=\mathcal{T},\gamma_2=\mathcal{T},\gamma_3=\mathcal{R},\gamma_4=\mathcal{T},
   \gamma_5=\mathcal{R}\},\\
   \land\{\gamma_1=\mathcal{T},\gamma_2=\mathcal{T},\gamma_3=\mathcal{R},\gamma_4=\mathcal{R},
   \gamma_5=\mathcal{T}\}\;\\
  \end{matrix}\!
 \right\rbrace\!\!\!
\end{equation}
\end{defn}
We'll also need to reason about $\neg q_1^5$.
Using disjunctive normal form we have:
$\neg q_1^5 \leftrightarrow$
\begin{equation}
\label{eq:notq15}
\bigvee\!\!
 \left\lbrace\!
  \begin{matrix}
   \land\{\gamma_1=\mathcal{T},\gamma_2=\mathcal{R},\gamma_3=\mathcal{T},\gamma_4=\mathcal{T},
   \gamma_5=\mathcal{R}\},\\
   \land\{\gamma_1=\mathcal{T},\gamma_2=\mathcal{R},\gamma_3=\mathcal{T},\gamma_4=\mathcal{R},
   \gamma_5=\mathcal{T}\},\vspace{0pt}\\
   \land\{\gamma_1=\mathcal{T},\gamma_2=\mathcal{R},\gamma_3=\mathcal{R},\gamma_4=\mathcal{T},
   \gamma_5=\mathcal{T}\},\\
   \land\{\gamma_1=\mathcal{R},\gamma_2=\mathcal{T},\gamma_3=\mathcal{R},\gamma_4=\mathcal{T},
   \gamma_5=\mathcal{T}\},\\
   \land\{\gamma_1=\mathcal{R},\gamma_2=\mathcal{R},\gamma_3=\mathcal{T},\gamma_4=\mathcal{T},
   \gamma_5=\mathcal{T}\}\;\\
  \end{matrix}\!
 \right\rbrace\!\!\!
\end{equation}
Adding $\gamma_1=\mathcal{R}$ possibilities (eq.\,\eqref{eq:gamma1IsR}) to $q_1^5$ and 
$\neg q_1^5$ gives
\begin{defn}%[$q_{15}^R$]
\label{def:q15R}
$q_{15}^R \coloneqq$
\begin{equation}
\label{eq:q15R}
\bigvee\!\!
 \left\lbrace\!
  \begin{matrix}
   \land\{\gamma_1=\mathcal{R},\gamma_2=\mathcal{T},\gamma_3=\mathcal{R},\gamma_4=\mathcal{T},
   \gamma_5=\mathcal{T}\},\\
   \land\{\gamma_1=\mathcal{R},\gamma_2=\mathcal{R},\gamma_3=\mathcal{T},\gamma_4=\mathcal{T},
   \gamma_5=\mathcal{T}\},\\
   \land\{\gamma_1=\mathcal{R},\gamma_2=\mathcal{T},\gamma_3=\mathcal{T},\gamma_4=\mathcal{T},
   \gamma_5=\mathcal{R}\},\\
   \land\{\gamma_1=\mathcal{R},\gamma_2=\mathcal{T},\gamma_3=\mathcal{T},\gamma_4=\mathcal{R},
   \gamma_5=\mathcal{T}\},\\
   \land\{\gamma_1=\mathcal{T},\gamma_2=\mathcal{T},\gamma_3=\mathcal{T},\gamma_4=\mathcal{R},
   \gamma_5=\mathcal{R}\},\vspace{0pt}\\
   \land\{\gamma_1=\mathcal{T},\gamma_2=\mathcal{T},\gamma_3=\mathcal{R},\gamma_4=\mathcal{T},
   \gamma_5=\mathcal{R}\},\\
   \land\{\gamma_1=\mathcal{T},\gamma_2=\mathcal{T},\gamma_3=\mathcal{R},\gamma_4=\mathcal{R},
   \gamma_5=\mathcal{T}\}\;\\
  \end{matrix}\!
 \right\rbrace\!\!
\end{equation}
\end{defn}
\begin{defn}%[$\bar{q}_{15}^R$]
\label{def:notq15R}
$\bar{q}_{15}^R \coloneqq$
\begin{equation}
\label{eq:notq15R}
\bigvee\!\!
 \left\lbrace\!
  \begin{matrix}
   \land\{\gamma_1=\mathcal{T},\gamma_2=\mathcal{R},\gamma_3=\mathcal{T},\gamma_4=\mathcal{T},
   \gamma_5=\mathcal{R}\},\\
   \land\{\gamma_1=\mathcal{T},\gamma_2=\mathcal{R},\gamma_3=\mathcal{T},\gamma_4=\mathcal{R},
   \gamma_5=\mathcal{T}\},\vspace{0pt}\\
   \land\{\gamma_1=\mathcal{T},\gamma_2=\mathcal{R},\gamma_3=\mathcal{R},\gamma_4=\mathcal{T},
   \gamma_5=\mathcal{T}\},\\
   \land\{\gamma_1=\mathcal{R},\gamma_2=\mathcal{T},\gamma_3=\mathcal{R},\gamma_4=\mathcal{T},
   \gamma_5=\mathcal{T}\},\\
   \land\{\gamma_1=\mathcal{R},\gamma_2=\mathcal{R},\gamma_3=\mathcal{T},\gamma_4=\mathcal{T},
   \gamma_5=\mathcal{T}\},\\
   \land\{\gamma_1=\mathcal{R},\gamma_2=\mathcal{T},\gamma_3=\mathcal{T},\gamma_4=\mathcal{T},
   \gamma_5=\mathcal{R}\},\\
   \land\{\gamma_1=\mathcal{R},\gamma_2=\mathcal{T},\gamma_3=\mathcal{T},\gamma_4=\mathcal{R},
   \gamma_5=\mathcal{T}\}\;\\
  \end{matrix}\!
 \right\rbrace\!\!
\end{equation}
\end{defn}

$q_{15}^R$ and $\bar{q}_{15}^R$ are then equally strong, and their construction provides a
%good
solution:
\begin{solution}
\label{solution:5gods}
A solution to the $5$ gods puzzle with 2 random and 3 truthful gods 
using $4.15$ questions exists.
\end{solution}
\begin{proof}
Put $t_q(q_1^5,\gamma_1)$ to $\gamma_1$ and consider the possible cases:
\begin{description}[itemsep=2pt,topsep=3pt,leftmargin=0.6em]
%\vspace{-0.55em}
%\setlength{\itemsep}{0pt}
\item[Case \(t(q_1^5,\gamma_1)\):] \hspace{0.0em}Then $q_{15}^R$ (eq.\,\eqref{eq:q15R}) 
holds by its construction, 
and theorem \ref{theorem:templateWorks}. 
%Since

Given 
$q_{15}^R$, 
$\gamma_2$ is most likely to be non-random,
and we'll ask her next.

We'll again aim to split the remaining possibilities in two equally large parts, 
and with gods likely to be non-random on both sides.
%along the larger gap??? in eq.\,\eqref{eq:q15R}. 
Let
\begin{defn}%[$q_2^5$]
\label{def:q25}
$q_{2}^5 \coloneqq$
\begin{equation}
\label{eq:q25}
\bigvee\!\!
 \left\lbrace\!
  \begin{matrix}
   \land\{\gamma_1=\mathcal{R},\gamma_2=\mathcal{R},\gamma_3=\mathcal{T},\gamma_4=\mathcal{T},
   \gamma_5=\mathcal{T}\},\\
   \land\{\gamma_1=\mathcal{R},\gamma_2=\mathcal{T},\gamma_3=\mathcal{T},\gamma_4=\mathcal{T},
   \gamma_5=\mathcal{R}\},\\
   \land\{\gamma_1=\mathcal{R},\gamma_2=\mathcal{T},\gamma_3=\mathcal{T},\gamma_4=\mathcal{R},
   \gamma_5=\mathcal{T}\},\\
   \land\{\gamma_1=\mathcal{T},\gamma_2=\mathcal{T},\gamma_3=\mathcal{T},\gamma_4=\mathcal{R},
   \gamma_5=\mathcal{R}\}\;\vspace{0pt}\\
  \end{matrix}\!
 \right\rbrace\!\!\!\!
\end{equation}
\end{defn}
Then $\neg q_{2}^5 \leftrightarrow$
\begin{equation}
\label{eq:notq25}
\bigvee\!\!
 \left\lbrace\!
  \begin{matrix}
   \land\{\gamma_1=\mathcal{R},\gamma_2=\mathcal{T},\gamma_3=\mathcal{R},\gamma_4=\mathcal{T},
   \gamma_5=\mathcal{T}\},\\
   \land\{\gamma_1=\mathcal{T},\gamma_2=\mathcal{T},\gamma_3=\mathcal{R},\gamma_4=\mathcal{T},
   \gamma_5=\mathcal{R}\},\\
   \land\{\gamma_1=\mathcal{T},\gamma_2=\mathcal{T},\gamma_3=\mathcal{R},\gamma_4=\mathcal{R},
   \gamma_5=\mathcal{T}\}\;\\
  \end{matrix}\!
 \right\rbrace\!\!\!\!\!
\end{equation}
There is no more $\gamma_2=\mathcal{R}$ possibility to add to $q_2^5$,
but there is one for $\neg q_2^5$:
\begin{defn}%[$\bar{q}_{25}^R$]
\label{def:notq25R}
$\bar{q}_{25}^R \coloneqq$
\begin{equation}
\label{eq:notq25R}
\bigvee\!\!
 \left\lbrace\!
  \begin{matrix}
   \land\{\gamma_1=\mathcal{R},\gamma_2=\mathcal{R},\gamma_3=\mathcal{T},\gamma_4=\mathcal{T},
   \gamma_5=\mathcal{T}\},\\
   \land\{\gamma_1=\mathcal{R},\gamma_2=\mathcal{T},\gamma_3=\mathcal{R},\gamma_4=\mathcal{T},
   \gamma_5=\mathcal{T}\},\\
   \land\{\gamma_1=\mathcal{T},\gamma_2=\mathcal{T},\gamma_3=\mathcal{R},\gamma_4=\mathcal{T},
   \gamma_5=\mathcal{R}\},\\
   \land\{\gamma_1=\mathcal{T},\gamma_2=\mathcal{T},\gamma_3=\mathcal{R},\gamma_4=\mathcal{R},
   \gamma_5=\mathcal{T}\}\;\\
  \end{matrix}\!
 \right\rbrace\!\!\!\!\!\!
\end{equation}
\end{defn}
$q_2^5$ and $\neg q_2^5$ are balanced, and we'll ask $\gamma_2$ about $q_2^5$ next.
%about $\gamma_3\neq\mathcal{R}$.
\begin{description}[itemsep=2pt,topsep=1pt,leftmargin=0.6em]
\item[Case $t(q_2^5,\gamma_2)$:]
Then $q_{2}^5$ (eq.\,\eqref{eq:q25}) holds, by construction. 
Therefore $\gamma_3\neq\mathcal{R}$ and is safe to ask. 

Asking $\gamma_3$ about $\gamma_4\neq\mathcal{R}$ and $\gamma_5\neq\mathcal{R}$
determine the rest of the gods.

This case used $4$ questions, and covered $4$ possibilities.

\item[Case $\neg t(q_2^5,\gamma_2)$:]
Then $\bar{q}_{25}^R$ (eq.\,\eqref{eq:notq25R}) holds, by construction. 

Next we'll go after $\gamma_4$, since it's likely that she isn't
$\mathcal{R}$.

We'll again aim to split the remaining possibilities in two equally large parts.
Let
\begin{defn}%[$q_3^5$]
\label{def:q35}
$q_{3}^5 \coloneqq$
\begin{equation}
\label{eq:q35}
%q_{3}^5 \!\!\coloneqq\!\!
\bigvee\!\!
 \left\lbrace\!
  \begin{matrix}
   \land\{\gamma_1=\mathcal{T},\gamma_2=\mathcal{T},\gamma_3=\mathcal{R},\gamma_4=\mathcal{T},
   \gamma_5=\mathcal{R}\},\\
   \land\{\gamma_1=\mathcal{T},\gamma_2=\mathcal{T},\gamma_3=\mathcal{R},\gamma_4=\mathcal{R},
   \gamma_5=\mathcal{T}\}\;\\
  \end{matrix}\!
 \right\rbrace\!\!\!
\end{equation}
\end{defn}
Then $\neg q_{3}^5 \leftrightarrow$
\begin{equation}
\label{eq:notq35}
%\!\neg q_{3}^5 \!\leftrightarrow\!\!
\bigvee\!\!
 \left\lbrace\!
  \begin{matrix}
   \land\{\gamma_1=\mathcal{R},\gamma_2=\mathcal{T},\gamma_3=\mathcal{R},\gamma_4=\mathcal{T},
   \gamma_5=\mathcal{T}\},\\
   \land\{\gamma_1=\mathcal{R},\gamma_2=\mathcal{R},\gamma_3=\mathcal{T},\gamma_4=\mathcal{T},
   \gamma_5=\mathcal{T}\}\;\\
  \end{matrix}\!
 \right\rbrace\!\!\!
\end{equation}
There is no more $\gamma_4=\mathcal{R}$ possibility to add to $q_3^5$,
but there is one for $\neg q_3^5$:
% \begin{defn}[$q_{35}^R$]
% \label{def:q35R}
% $q_{35}^R \coloneqq$
% \begin{equation}
% \label{eq:q35R}
% \bigvee\!\!
%  \left\lbrace\!
%   \begin{matrix}
%   \end{matrix}\!
%  \right\rbrace\!\!\!\!
% \end{equation}
% \end{defn}
\begin{defn}%[$\bar{q}_{35}^R$]
\label{def:notq35R}
$\bar{q}_{35}^R \coloneqq$
\begin{equation}
\label{eq:notq35R}
\bigvee\!\!
 \left\lbrace\!
  \begin{matrix}
   \land\{\gamma_1=\mathcal{R},\gamma_2=\mathcal{T},\gamma_3=\mathcal{R},\gamma_4=\mathcal{T},
   \gamma_5=\mathcal{T}\},\\
   \land\{\gamma_1=\mathcal{R},\gamma_2=\mathcal{R},\gamma_3=\mathcal{T},\gamma_4=\mathcal{T},
   \gamma_5=\mathcal{T}\},\\
   \land\{\gamma_1=\mathcal{T},\gamma_2=\mathcal{T},\gamma_3=\mathcal{R},\gamma_4=\mathcal{R},
   \gamma_5=\mathcal{T}\}\;\\
  \end{matrix}\!
 \right\rbrace\!\!\!
\end{equation}
\end{defn}
$q_3^5$ and $\neg q_3^5$ are as balanced as possible, and we'll ask $\gamma_4$ about $q_3^5$ next.
\begin{description}[itemsep=2pt,topsep=1pt,leftmargin=0.6em]
\item[Case $t(q_3^5,\gamma_4)$:]
Then $q_{3}^5$ (eq.\,\eqref{eq:q35}) holds, by construction. 
Therefore $\gamma_2\neq\mathcal{R}$ and is safe to ask. 

Asking $\gamma_2$ about $\gamma_4\neq\mathcal{R}$
determines the rest of the gods.

This case used $4$ questions, and covered $2$ possibilities.
\item[Case $\neg t(q_3^5,\gamma_3)$:]
Then $\bar{q}_{35}^R$ (eq.\,\eqref{eq:notq35R}) holds, by construction. 
Therefore $\gamma_5\neq\mathcal{R}$ and is safe to ask. 

Asking $\gamma_5$ about $\gamma_4\neq\mathcal{R}$, and if needed 
asking $\gamma_5$ about $\gamma_3\neq\mathcal{R}$,
determine the rest of the gods.

This case used $5$ questions for $2$ possibilities,
and $4$ questions for $1$ possibility.
\end{description}
\end{description}

%\enlargethispage{0.35\baselineskip}

\item[Case \(\neg t(q_1^5,\gamma_1)\):]
Then $\bar{q}_{15}^R$ (eq.\,\eqref{eq:notq15R}) holds, by construction.

We'll go after $\gamma_3$ next. Let
\begin{defn}%[$q_2^5$]
\label{def:q2bar5}
$q_{2}^{\bar{5}} \coloneqq$
\begin{equation}
\label{eq:q2bar5}
\bigvee\!\!
 \left\lbrace\!
  \begin{matrix}
   \land\{\gamma_1=\mathcal{T},\gamma_2=\mathcal{R},\gamma_3=\mathcal{T},\gamma_4=\mathcal{T},
   \gamma_5=\mathcal{R}\},\\
   \land\{\gamma_1=\mathcal{T},\gamma_2=\mathcal{R},\gamma_3=\mathcal{R},\gamma_4=\mathcal{T},
   \gamma_5=\mathcal{T}\},\\
   \land\{\gamma_1=\mathcal{R},\gamma_2=\mathcal{T},\gamma_3=\mathcal{T},\gamma_4=\mathcal{T},
   \gamma_5=\mathcal{R}\}\;\\
  \end{matrix}\!
 \right\rbrace\!\!\!\!
\end{equation}
\end{defn}
Then $\neg q_{2}^{\bar{5}} \leftrightarrow$
\begin{equation}
\label{eq:notq2bar5}
\bigvee\!\!
 \left\lbrace\!
  \begin{matrix}
   \land\{\gamma_1=\mathcal{T},\gamma_2=\mathcal{R},\gamma_3=\mathcal{T},\gamma_4=\mathcal{R},
   \gamma_5=\mathcal{T}\},\vspace{0pt}\\
   \land\{\gamma_1=\mathcal{R},\gamma_2=\mathcal{T},\gamma_3=\mathcal{T},\gamma_4=\mathcal{R},
   \gamma_5=\mathcal{T}\},\\
   \land\{\gamma_1=\mathcal{R},\gamma_2=\mathcal{R},\gamma_3=\mathcal{T},\gamma_4=\mathcal{T},
   \gamma_5=\mathcal{T}\},\\
   \land\{\gamma_1=\mathcal{R},\gamma_2=\mathcal{T},\gamma_3=\mathcal{R},\gamma_4=\mathcal{T},
   \gamma_5=\mathcal{T}\}\;\\
  \end{matrix}\!
 \right\rbrace\!\!\!\!\!
\end{equation}
Adding $\gamma_3=\mathcal{R}$ possibilities to $q_2^{\bar{5}}$ and 
$\neg q_2^{\bar{5}}$ gives
\begin{defn}
\label{def:q2bar5R}
$q_{2{\bar{5}}}^R \coloneqq$
\begin{equation}
\label{eq:q2bar5R}
\bigvee\!\!
 \left\lbrace\!
  \begin{matrix}
   \land\{\gamma_1=\mathcal{R},\gamma_2=\mathcal{T},\gamma_3=\mathcal{R},\gamma_4=\mathcal{T},
   \gamma_5=\mathcal{T}\},\\
   \land\{\gamma_1=\mathcal{T},\gamma_2=\mathcal{R},\gamma_3=\mathcal{T},\gamma_4=\mathcal{T},
   \gamma_5=\mathcal{R}\},\\
   \land\{\gamma_1=\mathcal{T},\gamma_2=\mathcal{R},\gamma_3=\mathcal{R},\gamma_4=\mathcal{T},
   \gamma_5=\mathcal{T}\},\\
   \land\{\gamma_1=\mathcal{R},\gamma_2=\mathcal{T},\gamma_3=\mathcal{T},\gamma_4=\mathcal{T},
   \gamma_5=\mathcal{R}\}\;\\
  \end{matrix}\!
 \right\rbrace\!\!\!\!\!\!
\end{equation}
\end{defn}
\begin{defn}%[$\bar{q}_{25}^R$]
\label{def:notq2bar5R}
$\bar{q}_{2{\bar{5}}}^R \coloneqq$
\begin{equation}
\label{eq:notq2bar5R}
\bigvee\!\!
 \left\lbrace\!
  \begin{matrix}
   \land\{\gamma_1=\mathcal{T},\gamma_2=\mathcal{R},\gamma_3=\mathcal{T},\gamma_4=\mathcal{R},
   \gamma_5=\mathcal{T}\},\vspace{0pt}\\
   \land\{\gamma_1=\mathcal{R},\gamma_2=\mathcal{T},\gamma_3=\mathcal{T},\gamma_4=\mathcal{R},
   \gamma_5=\mathcal{T}\},\\
   \land\{\gamma_1=\mathcal{R},\gamma_2=\mathcal{R},\gamma_3=\mathcal{T},\gamma_4=\mathcal{T},
   \gamma_5=\mathcal{T}\},\\
   \land\{\gamma_1=\mathcal{R},\gamma_2=\mathcal{T},\gamma_3=\mathcal{R},\gamma_4=\mathcal{T},
   \gamma_5=\mathcal{T}\},\\
   \land\{\gamma_1=\mathcal{T},\gamma_2=\mathcal{R},\gamma_3=\mathcal{R},\gamma_4=\mathcal{T},
   \gamma_5=\mathcal{T}\}\;\\
  \end{matrix}\!
 \right\rbrace\!\!\!\!\!\!
\end{equation}
\end{defn}
$q_2^{\bar{5}}$ and $\neg q_2^{\bar{5}}$ are balanced, and we'll ask $\gamma_3$ about
$q_2^{\bar{5}}$ next.
%about $\gamma_3\neq\mathcal{R}$.
\begin{description}[itemsep=2pt,topsep=1pt,leftmargin=0.6em]
\item[Case $t(q_2^{\bar{5}},\gamma_2)$:]
Then $q_{2{\bar{5}}}^R$ (eq.\,\eqref{eq:q2bar5R}) holds, by construction. 
Therefore $\gamma_4\neq\mathcal{R}$ and is safe to ask. 

Asking $\gamma_4$ about $\gamma_2\neq\mathcal{R}$ and $\gamma_3\neq\mathcal{R}$
determine the rest of the gods.

This case used $4$ questions, and covered $4$ possibilities.

\item[Case $\neg t(q_2^{\bar{5}},\gamma_2)$:]
Then $\bar{q}_{2{\bar{5}}}^R$ (eq.\,\eqref{eq:notq2bar5R}) holds, by construction. 
Therefore $\gamma_5\neq\mathcal{R}$ and is safe to ask. 

Ask $\gamma_5$ about $\gamma_4\neq\mathcal{R}$:
\begin{description}[itemsep=2pt,topsep=1pt,leftmargin=0.6em]
\item[If $\gamma_4\neq\mathcal{R}$,] then asking 
$\gamma_5$ about $\gamma_1\neq\mathcal{R}$, and if needed asking
$\gamma_5$ about $\gamma_3\neq\mathcal{R}$,
determine all the gods.

This case used $5$ questions for $2$ possibilities,
and $4$ questions for $1$ possibility.

\item[If $\gamma_4=\mathcal{R}$,] then asking 
$\gamma_5$ about $\gamma_1\neq\mathcal{R}$
determines all gods.

This case used $4$ questions, and covered $2$ possibilities.
\end{description}
\end{description}
\end{description}
\vspace{-0.45em}
\end{proof}
\vspace{-1.85em}

%\enlargethispage{1.10\baselineskip}

\subsection{Average number of questions used}
\label{sec:avg}

The number of questions used for each possibility is shown below \eqref{nOfQ}.
Some possibilities can be detected in more than one way.
When that can happen, 
%and when 
%%it's relevant,
%the number of questions differ,
the probability for each
question sequence is listed as well.
As can be seen, it's slightly better to try to hide
cases that can take $5$ questions among possibilities that have multiple ways to get detected
(e.g., $(\frac{1}{4}4+\frac{1}{4}5+\frac{1}{2}5) + 4 < \frac{4+5}{2}+\frac{4+5}{2}$);
see section \ref{sec:algo}.
\vspace{-0.2em}
\begin{equation}
\label{nOfQ}
 \begin{matrix}
  P\,|Qs| & \gamma_1    & \gamma_2    & \gamma_3    & \gamma_4    & \gamma_5    \\
  \frac{1}{4}4,\frac{1}{4}5,\frac{1}{2}5
          & \mathcal{R} & \mathcal{R} & \mathcal{T} & \mathcal{T} & \mathcal{T} \\
  \frac{1}{2}4,\frac{1}{2}4
          & \mathcal{R} & \mathcal{T} & \mathcal{T} & \mathcal{T} & \mathcal{R} \\[1pt]
  \frac{1}{2}4,\frac{1}{2}4
          & \mathcal{R} & \mathcal{T} & \mathcal{T} & \mathcal{R} & \mathcal{T} \\
  4       & \mathcal{T} & \mathcal{T} & \mathcal{T} & \mathcal{R} & \mathcal{R} \\
  4       & \mathcal{T} & \mathcal{T} & \mathcal{R} & \mathcal{T} & \mathcal{R} \\
  \frac{1}{2}4,\frac{1}{2}4
          & \mathcal{T} & \mathcal{T} & \mathcal{R} & \mathcal{R} & \mathcal{T} \\
  \frac{1}{2}5,\frac{1}{4}4,\mathit{\frac{1}{4}5}
          & \mathcal{R} & \mathcal{T} & \mathcal{R} & \mathcal{T} & \mathcal{T} \\
% 5       & \mathcal{R} & \mathcal{R} & \mathcal{T} & \mathcal{T} & \mathcal{T} \\
% 4       & \mathcal{T} & \mathcal{T} & \mathcal{R} & \mathcal{R} & \mathcal{T} \\

% 4       & \mathcal{R} & \mathcal{T} & \mathcal{R} & \mathcal{T} & \mathcal{T} \\
  4       & \mathcal{T} & \mathcal{R} & \mathcal{T} & \mathcal{T} & \mathcal{R} \\
  \frac{1}{2}4,\frac{1}{2}4 
          & \mathcal{T} & \mathcal{R} & \mathcal{R} & \mathcal{T} & \mathcal{T} \\
% 4       & \mathcal{R} & \mathcal{T} & \mathcal{T} & \mathcal{T} & \mathcal{R} \\

  4       & \mathcal{T} & \mathcal{R} & \mathcal{T} & \mathcal{R} & \mathcal{T} \\
% 4       & \mathcal{R} & \mathcal{T} & \mathcal{T} & \mathcal{R} & \mathcal{T} \\
% 5       & \mathcal{R} & \mathcal{R} & \mathcal{T} & \mathcal{T} & \mathcal{T} \\
% 5       & \mathcal{R} & \mathcal{T} & \mathcal{R} & \mathcal{T} & \mathcal{T} \\
% 4       & \mathcal{T} & \mathcal{R} & \mathcal{R} & \mathcal{T} & \mathcal{T} \\
 \end{matrix}
%\vspace{-0.15em}
\end{equation}
Thus, the average number of questions used to find a solution is
$ ( 2(\frac{1}{4}4 + \frac{1}{4}5 + \frac{1}{2}5) + 8\cdot 4 ) \;/\; 10 \;=\; 4.15 $.

% p. 6, l
%\enlargethispage{0.4\baselineskip}

\section{An algorithm for finding solutions}
\label{sec:algo}

You can derive an algorithm for solving the generalized problem from the solution in 
section \ref{sec:t3r2}. 

As before, the conclusions you can draw from a question $t_q(q,\gamma)$ to god $\gamma$ are 
     $q \vee \gamma\!=\!R$ from a ``yes''-answer, and 
$\neg q \vee \gamma\!=\!R$ from a ``no''-answer.

We will continue to write formulas in disjunctive normal form, DNF.
While DNF can often be a wasteful way of representing formulas and 
information, here it is useful since the entire representation is 
actively used.
%DNF is not wasteful here.
%%dnf, often wasteful but here it's useful as all info is used.

Disjunction is commutative and 
instead of allocating new memory for each new subproblem,
you can reorder the disjuncts in-place, and do not need 
%to allocate 
new memory for the problems, beyond the initial problem. 
%
%"Furthermore, because disjunction is commutative, we do not need to allocate new memory for each subproblem during the search. Instead, we can reorder the disjuncts in-place, partitioning the active problem space without requiring any memory beyond that of the initial problem."

To minimize the expected number of questions, as mentioned in 
%Regarding the comment in 
section \ref{sec:avg}, 
%about minimizing the average number of questions used,
let $\land\{\gamma_1=g_1, \ldots, \gamma_n=g_n\}$ be a possible god configuration disjunct. 
Then the average number of questions used to conclude
$\land\{\gamma_1=g_1, \ldots, \gamma_n=g_n\}$ is
$1/2^{r_1} q_1 + \ldots + 1/2^{r_k} q_k$, where $r_i$ is 
the number of random gods questioned on a path to a conclusion that this disjunct holds,
$q_i$ is the number of questions used for that particular conclusion,
and $k$ is the number of ways that the disjunct can be detected.
See probability factors $1/4$ and $1/2$ (and omitted $1/2^0$), and
number of questions $4$ and $5$, in \eqref{nOfQ}.
Hence, when we know that a god is non-random, handing out shorter detection paths
to disjuncts with smaller $r_i$ will result in a lower average number of questions used.

The algorithm has been 
implemented.\footnote{You
\label{fn:prog} 
can find the free and open-source program here:
%\url{https://drive.google.com/file/d/1mGnV2dojWzjLEBOMX4YemUAeCnIo3APc/view}} 
\url{https://github.com/DanielVallstrom/hardest}}
% Initial and cursory tests do indicate that trying to get a side to have a god with no chance of
% being random can be good, even at the expense of getting the sides a little unbalanced.
%
When run on our ``$0$-$3$-$2$'' problem from section \ref{sec:t3r2}, the program does find
a slightly better solution using $4.1375$ questions on average, and this is 
%probably
conjectured
optimal.
See appendix \ref{app:solverSolution} for a derivation of the solution from the solver.

% Aside from the fact that just the possible god configurations are superexponential,
% the problem itself is probably hard, too.

%Some 
%Initial and cursory tests also indicate that some semi-greedy heuristic maybe 
%doesn't work all that well. 
During a search you can estimate the final question average,
and abort a (sub-) search that is estimated to be above some upper bound. You can also
catch aborted searches at nodes that do seem promising, i.e.\ at nodes estimated
to be below some upper bound or threshold, and resume search at those promising nodes.
%However, as mentioned, even with some leeway, aborting unpromising searches will miss
%searches finding end result questions with lower averages, it seems.
%
% Gemini suggestion:
% "Preliminary experiments also suggest that employing a semi-greedy heuristic approach
%  is generally ineffective. During the search, one can estimate the final expected 
%  question average of a partial path and abort the exploration of any branch that 
%  exceeds a specific upper-bound threshold. Conversely, the algorithm can track and 
%  resume exploration at seemingly promising nodes that fall below this threshold. 
%  However, early estimates can be misleading. Even when allowing for a generous margin 
%  of error, prematurely terminating seemingly unpromising search paths frequently 
%  results in missing the specific configurations that ultimately lead to optimal solutions."

Randomization is implemented. For example, disjuncts can be shuffled.
%Support for e.g.\ shuffling disjuncts is implemented.
With the support of undoing and resuming sub-searches at any node,
% for the abort heuristic,
you can also repeat any sub-search a number of times, and take the best estimated
result.

% Also implemented is incorporating randomness, ubiquitously. 
% Tests indicate, as expected, that this is good. 
% You can also repeat any (sub-) search a number of times, and take the best estimated
% result. You would expect this to be useful too, and cursory tests 
% bear that out.
% %%indicate that. 
% %However, currently, this is implemented efficiently only for the first few
% %top levels.
% % % <obsolete>However, that taking the best of repeated searches fails to some extent to show a
% % % clear advantage might be due to the large overhead in the implementation to enable
% % % undoing (sub-) searches.

To minimize the risk of asking random gods, you can,
when preparing a question, 
%explicit and 
deliberately swap disjuncts between positive and
negative sides. 
%For example, if, say, $\gamma_i$ and $\gamma_j$
%are to be asked in the positive respectively negative case,
%and $\gamma_i$ currently is random in $20$ disjuncts in the positive case,
%and $\gamma_j$ currently is random in $30$ disjuncts in the negative case,
%you may be able to swap disjuncts between the positive and negative sides
%so that, say, 
%$\gamma_i$ will be random in $10$ disjuncts in the positive case, and
%$\gamma_j$ vill be random in $15$ disjuncts in the negative case.
Let $\gamma_i$ be the god to be asked after a positive answer, and 
let $\gamma_j$ be the god to be asked after a negative answer.
Then we can choose $\gamma_i$ and $\gamma_j$ so that, after swapping, 
the number of disjuncts where $\gamma_i$ and $\gamma_j$ are random in the
positive respectively 
%the 
negative case is minimized.
% is random in the positive case, plus 
% the number of disjuncts where $\gamma_j$ is random in the negative case,
% is minimized. 
% Although finding
% optimal $\gamma_i$ and $\gamma_j$, and then doing the swapping, is costly, 
% the total runtime will often
% be quite a bit shorter because optimal $\gamma_i$ and $\gamma_j$ will make
% the solution and search tree much smaller.

You can also swap disjuncts to balance the sides so that
$\gamma_i$ and $\gamma_j$ have the same risk of being random,
even though it doesn't lower the sum of the risks of asking random gods.
When the number of disjuncts to split is odd, and the risks of $\gamma_i$ and $\gamma_j$
being random differ minimally,
% number of random gods on the
%sides differ by $1$, 
you can try to put the extra disjunct and 
%random god
the extra risk
on different sides, to balance the strengths of the sides.

Initial 
%and cursory 
tests 
%do 
indicate that trying to get a side to have a god with no chance of
being random is sometimes good, even at the expense of getting the sides a little unbalanced.

You can milk a solution by basing a new search on that solution, 
but e.g.\ altering --- increasing or decreasing --- the number of  
repeated sub-searches at a node that you take the best estimated
result of.
Some form of milking has been implemented and tests indicate that
it often finds 
%small 
improvements.
%, but only small ones.

% Maybe the failure of a simple, semi-greedy, heuristic to provide clear improvements indicates that
% the underlying optimization problem is hard, 
% beyond the fact that 
% the search space is superexponential.
% %just the number of possible god configuration possibilities is exponential. 
% Or maybe some AI based heuristic can be good.

% p. 6, r
%\enlargethispage{0.4\baselineskip}

\section{Generalizations}
%\section{<Hardness> classification}
You can generalize the problem even more, to allow for any number of 
non-random
god types, 
%similar to the true and false gods, 
as long as you can deduce something from their answers. 
Similarly, 
%Likewise,
you can have more random god types, not just the one $\mathcal{R}$ type.
Further,
random gods function like noisy sources. 
%echoing principles of fault-tolerant computing \cite{pelc2002searching}.
See definition \ref{def:gen}.
\begin{defn}[The noisy source identification problem]
%[The noisy problem]
\label{def:gen}
Let $\mathcal{T}_1,\ldots,\mathcal{T}_i$ be 
%$i$ 
%true 
non-random
types of sources. 
Let $\mathcal{R}_1,\ldots,\mathcal{R}_j$ be noisy source types. 
Let $s_1,\ldots,s_n$ be sources, each of some unknown type.
% that is unknown.
%Let $s^{T_k}_1,\ldots,s^{T_k}_m$ be $m$ sources of type $T_k$. 
%Let $s^{R_k}_1,\ldots,s^{R_k}_n$ be $n$ sources of type $R_k$.
Noisy sources will answer questions randomly, while 
%true 
non-random sources will answer 
correctly.\footnote{As 
we saw in the puzzle, some of the non-random types could have their sources invert their answers, i.e., lie.
And we don't need to understand the sources; we just need to know that one word means yes or no. But 
that is not germane to the computational problem.}
%For any source $s$, we don't know what type it is. 
%Let 
Given are also 
$|\mathcal{T}_1|,\ldots,|\mathcal{T}_i|$ and 
$|\mathcal{R}_1|,\ldots,|\mathcal{R}_j|$, 
the number of sources of each type,
with 
$|\mathcal{T}_1|+\ldots +|\mathcal{T}_i| + 
 |\mathcal{R}_1|+\ldots +|\mathcal{R}_j| = n$.
The problem is to identify which type each source is,
using yes or no questions.
%with as few expected yes or no questions as possible.
\end{defn}

%%\section{Upper bounds}
%\section{Computed upper bounds}
%
%\subsection{A start on a classification of the generalized problem}
%
%
% The reproduction commands for the below bounds might very well be out of date. For
% updated bounds and reproduction commands, see the github repo.
%
% Bounds were found with solver version 0.15.5, and a few with 0.15.4. Flags and options
% listed here for 0.15.4 should make it so that 0.15.5 ought to be able to replicate the
% results as is.
%   For fast, one run, replication, add the seed listed, and possibly the upper bound (-u). 
% E.g., given the entry:
% "./hardest -f 2 -t 5 -r 2 -b 8 -B 0:9 -H 1 -i 7 -S 7 -a 1.02 -e 1.01 -u 9.91 -i 1000
% version 0.15.5
% 9.957672
% Start state: seed: 2318762516960874543, upper bound (-u): 9.9100000000000001"
% you can replicate the result, i.e. the upper bound 9.957672, quickly, with the command:
% "./hardest -f 2 -t 5 -r 2 -b 8 -B 0:9 -H 1 -i 7 -S 7 -a 1.02 -e 1.01 -u 9.91 -i 1000
%  -s 2318762516960874543 -i 0"
% (The -u value doesn't need updating; it never does with option -H 1 since it's biased low.)
\begin{table}
% p. 6, r
%\enlargethispage{0.2\baselineskip}
%{S[table-format=3.2, detect-all]}
%r@{.}l}
\vspace{-9pt}
%@S[table-format=2.6,detect-all]
\begin{tabular}{c@{\hspace{5pt}}c@{\hspace{5pt}}c@{\hspace{11pt}}r@{.}l}
  %\toprule
  $|\mathcal{F}|$  &  $|\mathcal{T}|$ & $|\mathcal{R}|$ & \multicolumn{2}{c}{\hspace{-24pt}$|Qs|$} \\[4pt]
  %\midrule
  0 & 2 & 1 & $\mathbf{2}$&$\mathbf{0}$ \\ 
  %\multicolumn{1}{B{.}{.}{-1}}{2.0} \\
  %{\bfseries 2.0} \\ 
  1 & 1 & 1 &  $\mathbf{3}$&$\mathbf{0}$ \\[4pt]
  
  0 & 3 & 1 &  $\mathbf{2}$&$\mathbf{375}$ \\ 
  1 & 2 & 1 &  $\mathbf{3}$&$\mathbf{916667}$ \\[4pt]

  0 & 4 & 1 &  $\mathbf{2}$&$\mathbf{6}$ \\
  1 & 3 & 1 &  $\mathbf{4}$&$\mathbf{6}$ \\
  2 & 2 & 1 &  $\mathbf{5}$&$\mathbf{133333}$ \\[1pt]
  0 & 3 & 2 &  $\mathit{4}$&$\mathit{13750}$ \\
  1 & 2 & 2 &  $\mathit{5}$&$\mathit{683333}$ \\[4pt]

  0 & 5 & 1 &  $\mathbf{2}$&$\mathbf{833333}$ \\
  1 & 4 & 1 &  $\mathbf{5}$&$\mathbf{1}$ \\
  2 & 3 & 1 &  $\mathbf{6}$&$\mathbf{1}$ \\[1pt]
  0 & 4 & 2 &  $\mathit{4}$&$\mathit{4}$ \\
  1 & 3 & 2 &  $\mathit{6}$&$\mathit{4}$ \\
  2 & 2 & 2 &  $\mathit{7}$&$\mathit{055556}$ \\[4pt]

  0 & 6 & 1 &  $\mathbf{3}$&$\mathbf{0}$ \\
  1 & 5 & 1 &  $\mathbf{5}$&$\mathbf{619048}$ \\
  2 & 4 & 1 &  $\mathbf{6}$&$\mathbf{923810}$ \\
  3 & 3 & 1 &  $\mathbf{7}$&$\mathbf{314286}$ \\[1pt]
  0 & 5 & 2 &  $\mathit{4}$&$\mathit{904762}$ \\
  1 & 4 & 2 &  $\mathit{7}$&$\mathit{161905}$ \\
  2 & 3 & 2 &  8&161905 \\[1pt]
%8.161905
%Start state: seed: 4604948742119086823, upper bound (-u): 8.1666666666666661
% v. 0.15.3
%./hardest -f 2 -t 3 -r 2 -b 4 20 -H 1 -g 3 -i 3000 -S 6 -U 4
  
  0 & 4 & 3 &  6&232143 \\
% ./hardest -f 0 -t 4 -r 3 -b 5 -B 0:6 -k 20 -H 1 -g 3 -i 70000 -S 7 -U 2 -a 1.02 -e 1.01 -u 6.148 -W 0.0
% v. 0.15.4
%6.232143
%Start state: seed: 8347231925438940503, upper bound (-u): 6.1479999999999997

  1 & 3 & 3 &  8&303850 \\
% 0.15.4
% ./hardest -f 1 -t 3 -r 3 -b 5 -B 0:6 -H 1 -i 7000 -S 7 -U 2 -a 1.02 -e 1.01 
%8.393973
%Start state: seed: 2557919246952201673, upper bound (-u): 8.4171875000000007
%Best estimated result for the positive case, depth 0: 8.261910
%Estimated result: 8.297769

  2 & 2 & 3 &  8&927679 \\[4pt]
% 0.15.4
%./hardest -f 2 -t 2 -r 3 -b 5 -B 0:6 -H 1 -i 8 -S 7 -U 2 -a 1.02 -e 1.01 -u 12
%9.049702
%Start state: seed: 9505483977630049913, upper bound (-u): 9.0604166666666668
%Best estimated result for the positive case, depth 0: 8.861092
%Estimated result: 8.936161

  0 & 7 & 1 & $\mathbf{3}$&$\mathbf{1875}$   \\ 
  1 & 6 & 1 & $\mathbf{5}$&$\mathbf{982143}$ \\
  2 & 5 & 1 & $\mathbf{7}$&$\mathbf{601190}$ \\
  3 & 4 & 1 & $\mathbf{8}$&$\mathbf{296429}$ \\[1pt]
  
  0 & 6 & 2 & $\mathit{5}$&$\mathit{178571}$ \\
  1 & 5 & 2 & $\mathit{7}$&$\mathit{833333}$ \\
  2 & 4 & 2 & 9&107143 \\
%./hardest -f 2 -t 4 -r 2 -b 6 -B 0:7 -H 1 -i 7 -S 7 -a 1.02 -e 1.01 -u 9.03 -i 2000
%9.109524
%Start state: seed: 14103940161794951466, upper bound (-u): 9.0299999999999994
%Best estimated result for the positive case, depth 0: 9.104409
%Estimated result: 9.102782
% v. 0.15.4

  3 & 3 & 2 & 9&533929 \\[1pt]
% v. 0.15.4
%./hardest -f 3 -t 3 -r 2 -b 9 -B 0:10 -H 1 -i 7 -S 7 -a 1.02 -e 1.01 -u 9.42 -i 2000
%9.535714
%Start state: seed: 6422440067747821884, upper bound (-u): 9.4199999999999999
%Best estimated result for the positive case, depth 0: 9.494529
%Estimated result: 9.505837
  
  0 & 5 & 3 & 6&571429 \\ 
%./hardest -f 0 -t 5 -r 3 -b 6 -B 0:7 -H 1 -i 7 -S 7 -a 1.02 -e 1.01 -u 6.48 -i 4000
% v. 0.15.4
%6.571429
%Start state: seed: 2026205479363403605, upper bound (-u): 6.4800000000000004
%Best estimated result for the positive case, depth 0: 6.434314
%Estimated result: 6.492682

  1 & 4 & 3 & 9&041964 \\ 
% v. 0.15.4
% ./hardest -f 1 -t 4 -r 3 -b 6 -B 0:7 -H 1 -i 7 -S 7 -a 1.02 -e 1.01 -u 8.99 -i 20
% 9.058482
%Start state: seed: 16670210726696513633, upper bound (-u): 8.9900000000000002
%Best estimated result for the positive case, depth 0: 9.046457
%Estimated result: 9.049366

  2 & 3 & 3 & 10&075670 \\[4pt] 
%./hardest -f 2 -t 3 -r 3 -b 4 -B 0:5 -H 1 -i 7 -S 7 -a 1.02 -e 1.01 -u 10.0 -i 10
%0.15.4
%10.091741
%Start state: seed: 13683506237796025932, upper bound (-u): 10
%Best estimated result for the positive case, depth 0: 10.076242
%Estimated result: 10.078705
 
  0 & 8 & 1 & $\mathbf{3}$&$\mathbf{333333}$ \\ 
  1 & 7 & 1 & $\mathbf{6}$&$\mathbf{333333}$ \\ 
  2 & 6 & 1 & $\mathbf{8}$&$\mathbf{095238}$ \\ 
  3 & 5 & 1 & $\mathbf{9}$&$\mathbf{095238}$ \\ 
  4 & 4 & 1 & $\mathbf{9}$&$\mathbf{485714}$ \\[1pt]
   
  0 & 7 & 2 & $\mathit{5}$&$\mathit{527778}$ \\ 
  1 & 6 & 2 & 8&$\mathit{273810}$ \\ 
% optimal? reproduced many times. easy with current, 0.15.4+
% reproduced: hardest 0.15.5:
%./hardest -f 3 -t 4 -r 2 -b 8 -B 0:9 -H 1 -i 7 -S 7 -a 1.02 -e 1.01 -u 10.57 -i 5000 -s 16955508027985230318 -i 0
% 8.273810

  2 & 5 & 2 & 9&956349 \\ 
%./hardest -f 2 -t 5 -r 2 -b 8 -B 0:9 -H 1 -i 7 -S 7 -a 1.02 -e 1.01 -u 9.91 -i 1000
% version 0.15.5
% 9.957672
%Start state: seed: 2318762516960874543, upper bound (-u): 9.9100000000000001
%Best estimated result for the positive case, depth 0: 9.928100
%Estimated result: 9.940423

  3 & 4 & 2 & 10&691270 \\[1pt]
%./hardest -f 3 -t 4 -r 2 -b 8 -B 0:9 -H 1 -i 7 -S 7 -a 1.02 -e 1.01 -u 10.57 -i 5000
% version 0.15.5
%10.692857
%Start state: seed: 16955508027985230318, upper bound (-u): 10.57
%Best estimated result for the positive case, depth 0: 10.638483
%Estimated result: 10.662984

  0 & 6 & 3 & 7&$\mathit{059524}$ \\ 
%./hardest -f 0 -t 6 -r 3 -b 8 -B 0:9 -H 1 -i 7 -S 7 -a 1.02 -e 1.01 -u 7.01 -i 5000
% version 0.15.5
% 7.059524
%Start state: seed: 14692722917154841026, upper bound (-u): 7.0099999999999998
%Best estimated result for the positive case, depth 0: 7.043833
%Estimated result: 7.048165

  1 & 5 & 3 & 9&696429 \\ 
% version 0.15.5
% ./hardest -f 1 -t 5 -r 3 -b 7 -B 0:8 -H 1 -i 7 -S 7 -a 1.02 -e 1.01 -u 9.617 -i 50
%9.721230
%Start state: seed: 3194072321188563417, upper bound (-u): 9.6169999999999991
%Best estimated result for the positive case, depth 0: 9.656368
%Estimated result: 9.683711

  2 & 4 & 3 & 11&065675 \\
% version 0.15.5
%./hardest -f 2 -t 4 -r 3 -b 6 -B 0:7 -H 1 -i 7 -S 7 -a 1.02 -e 1.01 -u 11.071 -i 50
%11.076984
%Start state: seed: 8602938993698001102, upper bound (-u): 11.071
%Best estimated result for the positive case, depth 0: 11.059684
%Estimated result: 11.066060

  3 & 3 & 3 & 11&424702 \\%[1pt]
% version 0.15.5
% ./hardest -f 3 -t 3 -r 3 -b 4 -B 0:5 -H 1 -S 7 -a 1.02 -e 1.01 -u 11.32 -i 50
% 11.446726
%Start state: seed: 7256326286338398992, upper bound (-u): 11.32
%Best estimated result for the positive case, depth 0: 11.401360
%Estimated result: 11.419665

\end{tabular}\hspace{8pt}
\begin{tabular}{c@{\hspace{5pt}}c@{\hspace{5pt}}c@{\hspace{11pt}}r@{.}l}
%\begin{tabular}{c@{\hspace{5pt}}c@{\hspace{5pt}}c@{\hspace{8pt}}S[table-format=2.6]}
  $|\mathcal{F}|$  &  $|\mathcal{T}|$ & $|\mathcal{R}|$ & \multicolumn{2}{c}{\hspace{-24pt}$|Qs|$} \\[4pt]
%  $|\mathcal{F}|$  &  $|\mathcal{T}|$  &  $|\mathcal{R}|$  &  \hspace{-4pt}$|Qs|$  \\[4pt]
  0 & 5 & 4 & 8&456039 \\ 
% version 0.15.5
%./hardest -f 0 -t 5 -r 4 -b 4 -B 0:5 -H 1 -S 7 -a 1.02 -e 1.01 -u 8.37 -i 50
% 8.470703
%Start state: seed: 3387689341141999424, upper bound (-u): 8.370000000000001
%Best estimated result for the positive case, depth 0: 8.292930
%Estimated result: 8.332628

  1 & 4 & 4 & 11&048462 \\ 
% version 0.15.5
%./hardest -f 1 -t 4 -r 4 -b 2 -B 0:2 -H 1 -S 7 -a 1.02 -e 1.01 -u 10.647 -i 50 -k 6
%11.057425
%Start state: seed: 4556814118304299227, upper bound (-u): 10.646999999999998
%Best estimated result for the positive case, depth 0: 10.657085
%Estimated result: 10.722083
  
  2 & 3 & 4 & 12&119603 \\[4pt]
% version 0.15.5
%./hardest -f 2 -t 3 -r 4 -b 2 -B 0:2 -H 1 -S 7 -a 1.02 -e 1.01 -u 11.82 -i 5 -k 6
%12.136675
%Start state: seed: 10896610904631108835, upper bound (-u): 11.82
%Best estimated result for the positive case, depth 0: 11.715108
%Estimated result: 11.779576

  0 & 9 & 1 & $\mathbf{3}$&$\mathbf{5}$ \\ 
  1 & 8 & 1 & $\mathbf{6}$&$\mathbf{677778}$ \\ 
  2 & 7 & 1 & $\mathbf{8}$&$\mathbf{677778}$ \\ 
  3 & 6 & 1 & $\mathbf{9}$&$\mathbf{880952}$ \\ 
  4 & 5 & 1 & $\mathbf{10}$&$\mathbf{474603}$ \\[1pt] 

  0 & 8 & 2 & $\mathit{5}$&$\mathit{844444}$ \\ 
  1 & 7 & 2 & 8&$\mathit{844444}$ \\
% .15.5
%   ./hardest -f 1 -t 7 -r 2 -b 9 -B 0:9 -H 1 -S 7 -a 1.02 -e 1.01 -u 8.73 -i 600 -k 11
% 8.844444
%Start state: seed: 14556849504793105477, upper bound (-u): 8.7299999999999986
%Best estimated result for the positive case, depth 0: 8.778531
%Estimated result: 8.802750

  2 & 6 & 2 & 10&648413 \\ 
% version 0.15.5
%./hardest -f 2 -t 6 -r 2 -b 7 -B 0:8 -H 1 -S 7 -a 1.02 -e 1.01 -u 10.54 -i 100 -k 11
%10.650000
%Start state: seed: 11087965487634666963, upper bound (-u): 10.539999999999999
%Best estimated result for the positive case, depth 0: 10.604714
%Estimated result: 10.625569

  3 & 5 & 2 & 11&653175 \\ 
% version 0.15.5
%./hardest -f 3 -t 5 -r 2 -b 9 -B 0:9 -H 1 -S 7 -a 1.02 -e 1.01 -u 11.53 -i 1000 -k 11
%11.654365
%Start state: seed: 14946608536867708297, upper bound (-u): 11.529999999999999
%Best estimated result for the positive case, depth 0: 11.609182
%Estimated result: 11.631636

  4 & 4 & 2 & 11&978095 \\[1pt] 
% version 0.15.5
%./hardest -f 4 -t 4 -r 2 -b 9 -B 0:9 -H 1 -S 7 -a 1.02 -e 1.01 -u 11.858 -i 1000 -k 11
%11.978095
%Start state: seed: 7689625710256102130, upper bound (-u): 11.858000000000001
%Best estimated result for the positive case, depth 0: 11.961964
%Estimated result: 11.975232

  0 & 7 & 3 & 7&$\mathit{425000}$ \\ 
%./hardest -f 0 -t 7 -r 3 -b 9 -B 0:9 -H 1 -S 7 -a 1.02 -e 1.01 -u 7.331 -i 1000 -k 11
% version 0.15.5
% 7.425000
%Start state: seed: 7098271598830889803, upper bound (-u): 7.3310000000000004
%Best estimated result for the positive case, depth 0: 7.355660
%Estimated result: 7.383776

  1 & 6 & 3 & 10&294048 \\
% version 0.15.5
%./hardest -f 1 -t 6 -r 3 -b 6 -B 0:7 -H 1 -S 7 -a 1.02 -e 1.01 -u 10.1987 -i 300 -k 11
% 10.298810
%Start state: seed: 18276219063542890244, upper bound (-u): 10.198699999999999
%Best estimated result for the positive case, depth 0: 10.256542
%Estimated result: 10.271596

  2 & 5 & 3 & 11&996379 \\
% version 0.15.5
% This illustrates milking a solution, by e.g. upping -b:
% ./hardest -f 2 -t 5 -r 3 -b 4 -B 0:5 -H 1 -S 7 -a 1.02 -e 1.01 -u 11.895 -i 0 -k 11 -s 17244908514822116033 -b 13
% 12.000992
%Best estimated result for the positive case, depth 0: 11.987927
%Estimated result: 11.991686
% Above is based on:
% ./hardest -f 2 -t 5 -r 3 -b 4 -B 0:5 -H 1 -S 7 -a 1.02 -e 1.01 -u 11.895 -i 8 -k 11
%12.002282
%Start state: seed: 17244908514822116033, upper bound (-u): 11.895000000000001
%Best estimated result for the positive case, depth 0: 12.002880
%Estimated result: 12.001948

  3 & 4 & 3 & 12&692381 \\[1pt] 
% version 0.15.5
%./hardest -f 3 -t 4 -r 3 -b 3 -B 0:4 -H 1 -S 7 -a 1.02 -e 1.01 -u 12.568 -i 80 -k 11
%12.698929
%Start state: seed: 8033834704288936178, upper bound (-u): 12.568
%Best estimated result for the positive case, depth 0: 12.664215
%Estimated result: 12.680489

  0 & 6 & 4 & 8&930729 \\
% version 0.15.5
% ./hardest -f 0 -t 6 -r 4 -b 4 -B 0:5 -H 1 -S 7 -a 1.02 -e 1.01 -u 8.807 -i 50 -k 11
% 8.941220
%Start state: seed: 248765155733712140, upper bound (-u): 8.8070000000000004
%Best estimated result for the positive case, depth 0: 8.768429
%Estimated result: 8.812000

  1 & 5 & 4 & 11&617990 \\
% version 0.15.5
%./hardest -f 1 -t 5 -r 4 -b 2 -B 0:2 -H 1 -S 7 -a 1.02 -e 1.01 -u 11.5256 -i 6 -k 11
% 11.646063
%Start state: seed: 4268585735434295512, upper bound (-u): 11.525600000000001
%Best estimated result for the positive case, depth 0: 11.459458
%Estimated result: 11.498258

  2 & 4 & 4 & 13&020278 \\
% version 0.15.5
%./hardest -f 2 -t 4 -r 4 -b 2 -B 0:2 -H 1 -S 7 -a 1.02 -e 1.01 -u 13.031 -i 1 -k 11
%13.025176
%Start state: seed: 8104028515989855274, upper bound (-u): 13.030999999999999
%Best estimated result for the positive case, depth 0: 12.905456
%Estimated result: 12.930257

  3 & 3 & 4 & 13&394975 \\[4pt]
% version 0.15.5
%./hardest -f 3 -t 3 -r 4 -b 1 -B 0:2 -H 1 -S 7 -a 1.02 -e 1.01 -u 13.28 -i 20 -k 11
%13.419905
%Start state: seed: 7166316949281471674, upper bound (-u): 13.279999999999999
%Best estimated result for the positive case, depth 0: 13.312786
%Estimated result: 13.339236

  0 & 10 & 1 & $\mathbf{3}$&$\mathbf{636364}$ \\ 
  1 & 9 & 1 & $\mathbf{6}$&$\mathbf{927273}$ \\ 
  2 & 8 & 1 & $\mathbf{9}$&$\mathbf{056566}$ \\ 
  3 & 7 & 1 & $\mathbf{10}$&$\mathbf{539394}$ \\ 
  4 & 6 & 1 & $\mathbf{11}$&$\mathbf{317749}$ \\ 
  5 & 5 & 1 & $\mathbf{11}$&$\mathbf{613276}$ \\[1pt]
  
  0 & 9 & 2 & $\mathit{6}$&$\mathit{054545}$ \\ 
  1 & 8 & 2 & 9&185859 \\ 
% version 0.15.5
% ./hardest -f 1 -t 8 -r 2 -b 9 -B 0:9 -H 1 -S 7 -a 1.02 -e 1.01 -u 9.1387 -i 8000 -k 11
% 9.189899
%Start state: seed: 16076386538645687829, upper bound (-u): 9.1386999999999983
%Best estimated result for the positive case, depth 0: 9.171151
%Estimated result: 9.176860

  2 & 7 & 2 & 11&204798 \\ 
% version 0.15.5  
%./hardest -f 2 -t 7 -r 2 -b 6 -B 0:7 -H 1 -S 7 -a 1.02 -e 1.01 -u 11.134 -i 1000 -k 11
%11.220960
%Start state: seed: 11004992265404454055, upper bound (-u): 11.133999999999999
%Best estimated result for the positive case, depth 0: 11.210718
%Estimated result: 11.208971

  3 & 6 & 2 & 12&474459 \\ 
% version 0.15.5  
% ./hardest -f 3 -t 6 -r 2 -b 7 -B 0:8 -H 1 -S 7 -a 1.02 -e 1.01 -u 12.36 -i 50 -k 11
%12.475974
%Start state: seed: 15954348099125908315, upper bound (-u): 12.360000000000001
%Best estimated result for the positive case, depth 0: 12.447222
%Estimated result: 12.456233

  4 & 5 & 2 & 13&049206 \\[1pt] 
% version 0.15.5  
%./hardest -f 4 -t 5 -r 2 -b 7 -B 0:8 -H 1 -S 7 -a 1.02 -e 1.01 -u 12.926 -i 140 -k 11
%13.051659
%Start state: seed: 4374285114512419273, upper bound (-u): 12.926
%Best estimated result for the positive case, depth 0: 13.046831
%Estimated result: 13.048323

  0 & 8 & 3 & 7&927273 \\ 
% version 0.15.5  
%./hardest -f 0 -t 8 -r 3 -b 8 -B 0:9 -H 1 -S 7 -a 1.02 -e 1.01 -u 7.82 -i 440 -k 11
%7.927273
%Start state: seed: 14297119481458267428, upper bound (-u): 7.8199999999999994
%Best estimated result for the positive case, depth 0: 7.829141
%Estimated result: 7.879797

  1 & 7 & 3 & 10&968939 \\
% version 0.15.5  
%./hardest -f 1 -t 7 -r 3 -b 3 -B 0:4 -H 1 -S 7 -a 1.02 -e 1.01 -u 10.868 -i 400 -k 11
%10.975000
%Start state: seed: 4106450522491929655, upper bound (-u): 10.868
%Best estimated result for the positive case, depth 0: 10.931839
%Estimated result: 10.965712

  2 & 6 & 3 & 12&767532 \\
% version 0.15.5  
%./hardest -f 2 -t 6 -r 3 -b 3 -B 0:4 -H 1 -S 7 -a 1.02 -e 1.01 -u 12.642 -i 100 -k 11
%12.769913
%Start state: seed: 15883515624880776119, upper bound (-u): 12.641999999999999
%Best estimated result for the positive case, depth 0: 12.709898
%Estimated result: 12.729604

  3 & 5 & 3 & 13&769697 \\ 
% version 0.15.5 
%./hardest -f 3 -t 5 -r 3 -b 1 -B 0:1 -H 1 -S 7 -a 1.02 -e 1.01 -u 13.634 -i 300 -k 11 -s 543823951367452778 -i 0
%13.773593
%Best estimated result for the positive case, depth 0: 13.749958
%Estimated result: 13.757156
 
  4 & 4 & 3 & 14&069199 \\[1pt]
% version 0.15.5  
%./hardest -f 4 -t 4 -r 3 -b 3 -B 0:4 -H 1 -S 7 -a 1.02 -e 1.01 -u 13.948 -i 16 -k 11
%14.081667
%Start state: seed: 7124894593297389499, upper bound (-u): 13.947999999999999
%Best estimated result for the positive case, depth 0: 14.075596
%Estimated result: 14.076306

  0 & 7 & 4 & 9&275758 \\ 
% version 0.15.5  
%./hardest -f 0 -t 7 -r 4 -b 1 -B 0:2 -H 1 -S 7 -a 1.02 -e 1.01 -u 9.214 -i 5000 -k 11
%9.292424
%Start state: seed: 14022505586819668448, upper bound (-u): 9.2139999999999986
%Best estimated result for the positive case, depth 0: 9.217271
%Estimated result: 9.259523

  1 & 6 & 4 & 12&232089 \\
% version 0.15.5  
% ./hardest -f 1 -t 6 -r 4 -b 1 -B 0:2 -H 1 -S 7 -a 1.02 -e 1.01 -u 12.146 -i 250 -k 11 
%12.247592
%Start state: seed: 3727540390642616461, upper bound (-u): 12.145999999999999
%Best estimated result for the positive case, depth 0: 12.199245
%Estimated result: 12.210973

  2 & 5 & 4 & 13&915627 \\ 
% version 0.15.5  
%./hardest -f 2 -t 5 -r 4 -b 1 -B 0:1 -H 1 -S 7 -a 1.02 -e 1.01 -u 13.787 -i 70 -k 10
%13.924281
%Start state: seed: 2868635322456824170, upper bound (-u): 13.786999999999999
%Best estimated result for the positive case, depth 0: 13.810699
%Estimated result: 13.840193

% This illustrates that larger -k can lower the run time, because the solution tree is smaller 
% (in this case, 9, 10, or so, all but one or a couple of gods, would still find best 
%  solution (not shown here)):
% ./hardest -f 2 -t 5 -r 4 -b 1 -B 0:1 -H 1 -S 7 -a 1.02 -e 1.01 -u 16.146 -i 0 -k 11
% |gods|=11 |fGods|=2 |tGods|=5 |rGods|=4 
% possibilities: 6930
% seed: 1773543570
% Result after one search: average number of questions asked to solve the problem: 13.985847
% Best estimated result for the positive case, depth 0: 14.022911
% Estimated result: 14.001152
% number of aborts caught: 0
% cpu time used, in seconds: 4.383
%
% ./hardest -f 2 -t 5 -r 4 -b 1 -B 0:1 -H 1 -S 7 -a 1.02 -e 1.01 -u 16.146 -i 0 -k 2
% Result after one search: average number of questions asked to solve the problem: 15.296295
% Best estimated result for the positive case, depth 0: 14.810383
% Estimated result: 14.880160
% cpu time used, in seconds: 116.686

  3 & 4 & 4 & 14&607397 \\[1pt]
% version 0.15.5  
%./hardest -f 3 -t 4 -r 4 -b 1 -B 0:1 -H 1 -S 7 -a 1.02 -e 1.01 -u 16.7827 -i 4 -k 11 -B 0:2 -B 1:2
% 14.631212
%Start state: seed: 4861299490506878948, upper bound (-u): 14.663203463203462
%Best estimated result for the positive case, depth 0: 14.584778
%Estimated result: 14.593616
 
  0 & 6 & 5 & 10&890225 \\
% version 0.15.5    
%./hardest -f 0 -t 6 -r 5 -b 1 -B 0:1 -H 1 -S 7 -a 1.02 -e 1.01 -u 10.49 -i 200 -k 11 -B 0:2 -B 1:2 -b 2
% 10.905836
%Start state: seed: 12443524816424922443, upper bound (-u): 10.49
%Best estimated result for the positive case, depth 0: 10.474120
%Estimated result: 10.534805

  1 & 5 & 5 & 13&731258 \\
% version 0.15.5  
%./hardest -f 1 -t 5 -r 5 -b 1 -B 0:1 -H 1 -S 7 -a 1.02 -e 1.01 -u 13.46 -i 50 -k 11
%13.767661
%Start state: seed: 8581647508762566537, upper bound (-u): 13.460000000000001
%Best estimated result for the positive case, depth 0: 13.398255
%Estimated result: 13.439318

  2 & 4 & 5 & 15&152598 \\
% version 0.15.5  
% Another milking illustration:
%./hardest -f 2 -t 4 -r 5 -b 1 -B 0:1 -H 1 -S 7 -a 1.02 -e 1.01 -u 16.46 -i 0 -k 11 -s 1773556797 -b 2
% seed: 1773556797
%Result after one search: average number of questions asked to solve the problem: 15.170539
%Best estimated result for the positive case, depth 0: 14.733296
%Estimated result: 14.777993
%
% The above milks this finding:
% version 0.15.5  
%./hardest -f 2 -t 4 -r 5 -b 1 -B 0:1 -H 1 -S 7 -a 1.02 -e 1.01 -u 16.46 -i 50 -k 11
% seed: 1773556797
%Result after one search: average number of questions asked to solve the problem: 15.188674
%Best estimated result for the positive case, depth 0: 14.757113
%Estimated result: 14.801330

  3 & 3 & 5 & 15&555879 \\[4pt]
% version 0.15.5 
%./hardest -f 3 -t 3 -r 5 -b 0 -H 1 -S 7 -a 1.02 -e 1.01 -u 15.409 -i 10 -k 11 -B 0:1 -B 1:1 -B 2:1 -B 3:1 -B 4:1 -B 5:1 -B 6:1
%15.600163
%Start state: seed: 18253059765483928619, upper bound (-u): 15.409000000000001
%Best estimated result for the positive case, depth 0: 15.298040
%Estimated result: 15.332084

  0 & 11 & 1 & $\mathbf{3}$&$\mathbf{750000}$ \\ 
  1 & 10 & 1 & $\mathbf{7}$&$\mathbf{143939}$ \\ 
  2 & 9 & 1 & $\mathbf{9}$&$\mathbf{531818}$ \\ 
  3 & 8 & 1 & $\mathbf{11}$&$\mathbf{048990}$ \\ 
  4 & 7 & 1 & $\mathbf{12}$&$\mathbf{048990}$ \\ 

\end{tabular}
\vspace{-0.60em}
\caption{Upper bounds on the average number
of questions needed to solve the generalization of The Hardest Logic Puzzle Ever.
Values in \textbf{bold} are proven optimal. Values in \textit{italics} are conjectured optimal.
Values with decimals in italics could very well be optimal too.}
%The table was derived using the program from footnote \ref{fn:prog}.}
\label{tab:classification}
\end{table}
% p. 7, l
%\enlargethispage{0.1\baselineskip}

\begin{table}[ht]
% p. 6, r
%\enlargethispage{0.2\baselineskip}
\vspace{-9pt}
\begin{tabular}{c@{\hspace{5pt}}c@{\hspace{5pt}}c@{\hspace{11pt}}r@{.}l}
%\begin{tabular}{c@{\hspace{5pt}}c@{\hspace{5pt}}c@{\hspace{8pt}}S[table-format=2.6]}
  $|\mathcal{F}|$  &  $|\mathcal{T}|$ & $|\mathcal{R}|$ & \multicolumn{2}{c}{\hspace{-24pt}$|Qs|$} \\[4pt]
%\begin{tabular}{c@{\hspace{5pt}}c@{\hspace{5pt}}c@{\hspace{5pt}}S[table-format=2.6]}
%  $|\mathcal{F}|$  &  $|\mathcal{T}|$  &  $|\mathcal{R}|$  &  \hspace{-4pt}$|Qs|$  \\[4pt]
  5 & 6 & 1 & $\mathbf{12}$&$\mathbf{605700}$ \\[2pt]

  0 & 10 & 2 & $\mathit{6}$&$\mathit{272727}$ \\ 
  1 & 9 & 2 & 9&660606 \\
% Is this value optimal? This insance is easier for the solver to solve than 1-8-2. Why?
% This value might be optimal; the solver consistently finds it. 
% version 0.15.5   
%./hardest -f 1 -t 9 -r 2 -b 4 -B 0:5 -H 1 -S 7 -a 1.02 -e 1.01 -u 9.59 -i 1000 -k 12 -b 7
%9.660606
%Start state: seed: 7088352904154992786, upper bound (-u): 9.5899999999999999
%Best estimated result for the positive case, depth 0: 9.599313
%Estimated result: 9.622241

  2 & 8 & 2 & 11&839057 \\ 
% version 0.15.5 
%./hardest -f 2 -t 8 -r 2 -b 4 -B 0:5 -H 1 -S 7 -a 1.02 -e 1.01 -u 11.717 -i 1000 -k 12 -b 6
%11.839394
%Start state: seed: 2767003436672560399, upper bound (-u): 11.716999999999999
%Best estimated result for the positive case, depth 0: 11.805893
%Estimated result: 11.833121

  3 & 7 & 2 & 13&188826 \\ 
% version 0.15.5 
% ./hardest -f 3 -t 7 -r 2 -b 4 -B 0:5 -H 1 -S 7 -a 1.02 -e 1.01 -u 13.108 -i 500 -k 12
%13.203157
%Start state: seed: 15888626169369702493, upper bound (-u): 13.107999999999999
%Best estimated result for the positive case, depth 0: 13.193349
%Estimated result: 13.192352

  4 & 6 & 2 & 14&026659 \\
% version 0.15.5 
%./hardest -f 4 -t 6 -r 2 -b 4 -B 0:5 -H 1 -S 7 -a 1.02 -e 1.01 -u 13.8944 -i 300 -k 12 -b5
%14.030213
%Start state: seed: 17468815447494029456, upper bound (-u): 13.894400000000001
%Best estimated result for the positive case, depth 0: 14.029091
%Estimated result: 14.029672

  5 & 5 & 2 & 14&265753 \\[1pt]
% version 0.15.5 
%./hardest -f 5 -t 5 -r 2 -b 4 -B 0:5 -H 1 -S 7 -a 1.02 -e 1.01 -u 14.169 -i 100 -k 12 
%14.279161
%Start state: seed: 10408021407934587909, upper bound (-u): 14.169
%Best estimated result for the positive case, depth 0: 14.263205
%Estimated result: 14.265541

  0 & 9 & 3 & 8&195455 \\[0pt]
% it's possible that it's optimal
% version 0.15.6 
%./hardest -f 0 -t 9 -r 3 -H 0 -S 7 -c 6 -u 8.106 -i 99 -l 1.1 -b 9
%8.195455
%Start state: seed: 14629452660196371190, upper bound (-u): 8.1059999999999999
%Best estimated result for the positive case, depth 0: 8.144231
%Estimated result: 8.153846

  1 & 8 & 3 & 11&393939 \\[0pt]
% version 0.15.6
%./hardest -f 1 -t 8 -r 3 -H 0 -S 7 -c 6 -u 11.31 -i 199 -l 1.1 -b 4
%11.404040
%Start state: seed: 15721476607074475487, upper bound (-u): 11.31
%Best estimated result for the positive case, depth 0: 11.392197
%Estimated result: 11.395587
 
  2 & 7 & 3 & 13&414457 \\[0pt]
% version 0.15.6
%./hardest -f 2 -t 7 -r 3 -b 1 -B 0:1 -H 1 -S 7 -c 6 -u 13.3109 -i 59 -l 1.1 -b 2
%13.427841
%Start state: seed: 7164996095181785419, upper bound (-u): 13.3109
%Best estimated result for the positive case, depth 0: 13.389973
%Estimated result: 13.416132
 
  3 & 6 & 3 & 14&703220 \\[0pt]
% version 0.15.6
%./hardest -f 3 -t 6 -r 3 -b 1 -B 0:1 -H 1 -S 7 -c 6 -u 14.561 -i 19 -l 1.1 -B 0:2 -B 1:2 -B 2:2 -B 3:2
%14.705303
%Start state: seed: 410639136458424372, upper bound (-u): 14.561
%Best estimated result for the positive case, depth 0: 14.678356
%Estimated result: 14.687731

  4 & 5 & 3 & 15&231133 \\[1pt]
% version 0.15.6 
%./hardest -f 4 -t 5 -r 3 -b 1 -B 0:1 -H 1 -S 7 -c 6 -u 15.094 -i 99 -l 1.1
%15.239854
%Start state: seed: 1396675600036727563, upper bound (-u): 15.093999999999999
%Best estimated result for the positive case, depth 0: 15.224523
%Estimated result: 15.232182

  0 & 8 & 4 & 9&832702 \\[0pt]
% version 0.15.6 
%./hardest -f 0 -t 8 -r 4 -b 1 -B 0:2 -H 1 -S 7 -c 20 -u 9.72 -i 400 -l 1.2 -b 3
%9.840404
%Start state: seed: 17588120503035732877, upper bound (-u): 9.7199999999999989
%Best estimated result for the positive case, depth 0: 9.748541
%Estimated result: 9.779618

  1 & 7 & 4 & 12&932343 \\[0pt]
% version 0.15.6 
%./hardest -f 1 -t 7 -r 4 -b 1 -B 0:2 -H 1 -S 7 -c 20 -u 12.81 -i 80 -l 1.3
%12.950126
%Start state: seed: 17306329951023624122, upper bound (-u): 12.81
%Best estimated result for the positive case, depth 0: 12.851905
%Estimated result: 12.914195

  2 & 6 & 4 & 14&720797 \\[0pt]
% version 0.15.6 
%./hardest -f 2 -t 6 -r 4 -b 1 -B 0:2 -H 1 -S 7 -c 20 -u 14.59 -i 20 -l 1.3
% 14.733135
%Start state: seed: 11891906751021216175, upper bound (-u): 14.59
%Best estimated result for the positive case, depth 0: 14.679935
%Estimated result: 14.689364

  3 & 5 & 4 & 15&720333 \\[0pt]
% version 0.15.6 
%./hardest -f 3 -t 5 -r 4 -b 1 -B 0:2 -H 1 -S 7 -c 20 -u 15.59 -i 9 -l 1.2
%seed: 1773801667
%Result after one search: average number of questions asked to solve the problem: 15.746904
%Best estimated result for the positive case, depth 0: 15.731904
%Estimated result: 15.733173

  4 & 4 & 4 & 16&060301 \\[1pt]
% version 0.15.6 
%./hardest -f 4 -t 4 -r 4 -b 1 -B 0:1 -H 1 -S 7 -c 6 -u 15.93 -i 19 -l 1.1
%16.068750
%Start state: seed: 10150697525459360563, upper bound (-u): 15.93
%Best estimated result for the positive case, depth 0: 16.058809
%Estimated result: 16.065504

  0 & 7 & 5 & 11&226503 \\[0pt]
%./hardest -f 0 -t 7 -r 5 -b 1 -B 0:1 -H 1 -S 7 -c 20 -u 11.170 -i 200 -l 1.1
% version 0.15.6
%11.249620
%Start state: seed: 12712224546223165494, upper bound (-u): 11.17
%Best estimated result for the positive case, depth 0: 10.946618
%Estimated result: 11.004183

  1 & 6 & 5 & 14&226982 \\[0pt]
%./hardest -f 1 -t 6 -r 5 -b 1 -B 0:1 -c 20 -H 1 -S 7 -u 14.179 -i 5 -s 8041377118281470216
% version 0.15.6
%Result after 2 searches: average number of questions asked to solve the problem: 14.249496
%Start state: seed: 8041377118281470216, upper bound (-u): 14.179
%Best estimated result for the positive case, depth 0: 14.195467
%Estimated result: 14.198750

  2 & 5 & 5 & 15&900892 \\[0pt]
% version 0.15.5 
%./hardest -f 2 -t 5 -r 5 -b 0 -B 0:1 -H 1 -S 7 -a 1.02 -e 1.01 -u 15.871 -i 10 -k 12 -B 1:1 -B 2:1 -B 3:1 -B 4:1
%15.933610
%Start state: seed: 8968094322052781071, upper bound (-u): 15.871
%Best estimated result for the positive case, depth 0: 15.710757
%Estimated result: 15.746717

  3 & 4 & 5 & 16&635073 \\[9pt]
% version 0.15.5 
%./hardest -f 3 -t 4 -r 5 -b 1 -B 0:1 -H 1 -S 7 -a 1.02 -e 1.01 -u 16.429 -i 4 -k 12 
%16.677641
%Start state: seed: 5821956879256807125, upper bound (-u): 16.428999999999998
%Best estimated result for the positive case, depth 0: 16.506532
%Estimated result: 16.539214

  0 & 12 & 1 & $\mathbf{3}$&$\mathbf{846154}$ \\ 
  0 & 13 & 1 & $\mathbf{3}$&$\mathbf{928571}$ \\
  0 & 14 & 1 & $\mathbf{4}$&$\mathbf{0}$ \\
  0 & 15 & 1 & $\mathbf{4}$&$\mathbf{093750}$ \\
  0 & 16 & 1 & $\mathbf{4}$&$\mathbf{176471}$ \\
  0 & 17 & 1 & $\mathbf{4}$&$\mathbf{277778}$ \\
  0 & 18 & 1 & $\mathbf{4}$&$\mathbf{368421}$ \\

\end{tabular}\hspace{8pt}
\begin{tabular}{c@{\hspace{5pt}}c@{\hspace{5pt}}c@{\hspace{11pt}}r@{.}l}
  $|\mathcal{F}|$  &  $|\mathcal{T}|$ & $|\mathcal{R}|$ & \multicolumn{2}{c}{\hspace{-24pt}$|Qs|$} \\[4pt]
%\begin{tabular}{c@{\hspace{5pt}}c@{\hspace{5pt}}c@{\hspace{8pt}}S[table-format=2.6]}
%  $|\mathcal{F}|$  &  $|\mathcal{T}|$  &  $|\mathcal{R}|$  &  \hspace{-4pt}$|Qs|$  \\[4pt]
  0 & 19 & 1 & $\mathbf{4}$&$\mathbf{450000}$ \\
  0 & 20 & 1 & $\mathbf{4}$&$\mathbf{523810}$ \\  
  0 & 25 & 1 & $\mathbf{4}$&$\mathbf{807692}$ \\
  0 & 30 & 1 & $\mathbf{5}$&$\mathbf{0}$ \\
  0 & 35 & 1 & $\mathbf{5}$&$\mathbf{25}$ \\
  0 & 40 & 1 & $\mathbf{5}$&$\mathbf{463415}$ \\
  0 & 45 & 1 & $\mathbf{5}$&$\mathbf{630435}$ \\
  0 & 50 & 1 & $\mathbf{5}$&$\mathbf{764706}$ \\
  0 & 60 & 1 & $\mathbf{5}$&$\mathbf{967213}$ \\ 
  0 & 62 & 1 & $\mathbf{6}$&$\mathbf{0}$ \\ 
  0 & 70 & 1 & $\mathbf{6}$&$\mathbf{211268}$ \\
  0 & 80 & 1 & $\mathbf{6}$&$\mathbf{432099}$ \\
  0 & 90 & 1 & $\mathbf{6}$&$\mathbf{604396}$ \\  
  0 & 100 & 1 & $\mathbf{6}$&$\mathbf{742574}$ \\
  0 & 126 & 1 & $\mathbf{7}$&$\mathbf{0}$ \\  
  0 & 150 & 1 & $\mathbf{7}$&$\mathbf{311258}$ \\
  0 & 200 & 1 & $\mathbf{7}$&$\mathbf{731343}$ \\  
  0 & 250 & 1 & $\mathbf{7}$&$\mathbf{984064}$ \\
  0 & 254 & 1 & $\mathbf{8}$&$\mathbf{0}$ \\[1pt]
  
  0 & 11 & 2 & $\mathit{6}$&$\mathit{551282}$ \\
  0 & 12 & 2 & $\mathit{6}$&$\mathit{769231}$ \\
  0 & 13 & 2 & $\mathit{6}$&$\mathit{942857}$ \\
  0 & 14 & 2 & $\mathit{7}$&$\mathit{075000}$ \\
  0 & 15 & 2 & $\mathit{7}$&$\mathit{257353}$ \\
  0 & 16 & 2 & $\mathit{7}$&$\mathit{457516}$ \\
  0 & 17 & 2 & $\mathit{7}$&$\mathit{625731}$ \\
  0 & 18 & 2 & $\mathit{7}$&$\mathit{768421}$ \\
  0 & 19 & 2 & $\mathit{7}$&$\mathit{890476}$ \\
  0 & 20 & 2 & $\mathit{7}$&$\mathit{995671}$ \\

\end{tabular}
\vspace{-0.36em}
\caption{Upper bounds on the average number
of questions needed to solve the generalization of The Hardest Logic Puzzle Ever.}
\label{tab:tab2}
% p. 7, r
%\enlargethispage{0.1\baselineskip}
\end{table}

Having more types will make for many more 
%god 
%source
possible
configurations.
% possibilities. 
However, 
%the defining feature 
the central source of difficulty in these problems is the presence 
of 
at least one noisy source.
%random god type.
%of these problems is that there is at least one random god type.

The fact that you can fairly efficiently determine 
the state of affairs even
when just under half of 
%gods are random
sources are noisy
is perhaps remarkable.

\begin{theorem}
\label{theorem:gen}
A noisy source identification problem is solvable if and only if 
the number of random sources is strictly less than the number of non-random sources.
\end{theorem}
\noindent (The proof for the case with just
true, false, and random gods,
theorem \ref{theorem:nGodsSolvability}, generalizes to a proof for this theorem
(\ref{theorem:gen}).)

With 
%Since
%The fact that 
random gods functioning as noisy sources,
%provides 
%there is 
%a connection 
%are connections
%to fault-tolerant
results here 
%are related 
%to
echo
principles of 
fault-tolerant 
computing.\cite{preparata1967connection, pelc1998reliable, pelc2002searching}
%diks1997globally
%random gods 
%%function as 
%act like
%noisy sources, 
%%connecting <this> 
%%providing 
%which provides
%a connection
%to fault-tolerant computing.

\section{Computed upper bounds}

Tables \ref{tab:classification} and \ref{tab:tab2} show
upper bounds on the 
expected number 
%of questions required
%average number
of questions 
needed to solve the generalized problem for
various instances.
% with arbitrary number of 
% $\mathcal{F}$, $\mathcal{T}$, and $\mathcal{R}$ gods,
% %when the number of gods is $\leq 11$.
% for configurations totaling $n\leq 11$ gods.
% %with only the simplest problems having exact solutions, in bold.
% Table \ref{tab:tab2} shows upper bounds for 
% selected configurations with $n>11$.
%
$|Qs|$ for $k$ false gods and $m$ true gods
%, with $k$>$m$,
equals $|Qs|$ for $m$ false gods and $k$ true gods, for symmetry reasons
(and similarly for problems with more types).
All values were computed using the solver from footnote 
\ref{fn:prog}.\footnote{For
%The 
\label{fn:bounds}
an updated list of
%place hosting the solver also maintains a list of 
currently 
known bounds, 
along with instructions for reproducing them,
see:
\url{https://github.com/DanielVallstrom/hardest}
You are welcome to upload new bounds there.}

% This doesn' work, for some reason.
\if 01
% \newcolumntype{.}{D{.}{.}{-1}}
% %\makeatletter
% %\newcolumntype{B}[3]{>{\boldmath\DC@{#1}{#2}{#3}}c<{\DC@end}}
% \makeatletter
% \newcolumntype{I}[3]{>{\mathit\DC@{#1}{#2}{#3}}c<{\DC@end}}
% \makeatletter
% \newcolumntype{N}[3]{>{\DC@{#1}{#2}{#3}}c<{\DC@end}}
% \makeatother
% \begin{table}
% % p. 6, r
% %\enlargethispage{0.2\baselineskip}
% %{S[table-format=3.2, detect-all]}
% %r@{.}l}
% \vspace{-2pt}
% %@S[table-format=2.6,detect-all]
% \begin{tabular}{c@{\hspace{5pt}}c@{\hspace{5pt}}c@{\hspace{-30pt}}d{2.6}}
%   %\toprule
%   $|\mathcal{F}|$  &  $|\mathcal{T}|$  &  $|\mathcal{R}|$  &  \multicolumn{1}{c}{\hspace{-0pt}$|Qs|$}  \\[4pt]
%   %\midrule
% \end{tabular}
% \vspace{-0.36em}
% \caption{Upper bounds on the average number
% of questions needed to solve the generalization of ``The Hardest Logic Puzzle Ever''.}
% \label{tab:tab2}
% \end{table}
%
\fi

With no loss of generality, we asked the first god, $\gamma_1$, first in section \ref{sec:t3r2},
and then proceeded to ask the second and third gods, $\gamma_2$ and $\gamma_3$. The solver follows the same
pattern, but there the gods are called $g_0$, $g_1$, and $g_2$.
Appendix \ref{app:oneR} shows why the solver will find optimal solutions to problems where
it is possible to swap disjuncts so that the gods
%, $g_1$ and $g_2$, 
being asked after the first god have no chance of being random. 
Problems where the second and third asked gods won't be random include 
all problems with just one random god. In particular, we have the following theorem:
\begin{theorem}
\label{theorem:oneR}
The optimal number of questions for problems with $1$ random god, $2^n-2$ true gods, 
and $0$ false gods, is $n$.
\end{theorem}
\vspace{-0.25em}

\section{On the complexity of the generalized problem}

From a naive perspective, especially given the presence of random sources,
the possible questions to ask might seem infinite.
However, a model-theoretic argument shows that it is enough to
ask about the possible source configurations:
if two questions are true in the same configurations, then they are
logically equivalent.
%(see the completeness proofs here (\ref{} and \ref{})).
%However, a simple model-theoretic argument shows that it is sufficient to ask only about the possible source configurations: 
%gemini:
%"since any two questions that hold true in the exact same set of configurations are logically equivalent, every distinct question corresponds to a subset of the valid configurations"
%

As in definition \ref{def:gen},
let $\mathcal{T}_1,\ldots,\mathcal{T}_i$ be non-random types, 
let $\mathcal{R}_1,\ldots,\mathcal{R}_j$ be noisy source types, 
let $s_1,\ldots,s_n$ be sources, and
let 
$|\mathcal{T}_1|,\ldots,|\mathcal{T}_i|$ and 
$|\mathcal{R}_1|,\ldots,|\mathcal{R}_j|$ 
denote the number of sources of each type.
Then the number of possible initial configurations, $C$, is given by
% \begin{equation}
% \begin{gathered}
% %\frac{n\cdot\ldots\cdot(n-|T_1|+1)}{|T_1|!} \cdot \ldots \cdot
% \frac{n!}{|\mathcal{T}_1|!(n-|\mathcal{T}_1|)!} \cdot \ldots \cdot  \\
% \frac{(n-|\mathcal{T}_1|-\ldots-|\mathcal{T}_{i-1}|)!}
%       {|\mathcal{T}_i|!
%        (n-|\mathcal{T}_1|-\ldots-|\mathcal{T}_{i-1}|-|\mathcal{T}_i|)!} \cdot  \\
% \frac{(n-|\mathcal{T}_1|-\ldots-|\mathcal{T}_{i}|)!}{|\mathcal{R}_1|!(n-|\mathcal{T}_1|-\ldots-|\mathcal{T}_{i}|-|\mathcal{R}_1|)!} \cdot \ldots \cdot  \\
% \frac{(n-|\mathcal{T}_1|-\ldots-|\mathcal{T}_{i}| -
%           |\mathcal{R}_1|-\ldots-|\mathcal{R}_{j-2}|)!}
%       {\begin{aligned}|\mathcal{R}_{j-1}|!
%           (n-|\mathcal{T}_1|-\ldots-|\mathcal{T}_{i}| -
%           |\mathcal{R}_1|-\ldots-&|\mathcal{R}_{j-2}| -  \\
%           & |\mathcal{R}_{j-1}|)!\end{aligned}} 
% \end{gathered}
% \end{equation}
%which is 
the multinomial coefficient:
$C=$
\begin{equation*}
\label{eq:multinomial}
\binom{n}{|\mathcal{T}_1|,\ldots,|\mathcal{T}_i|,|\mathcal{R}_1|,\ldots,|\mathcal{R}_j|} = 
\end{equation*}
\begin{equation}
\label{eq:multiFac}
\frac{n!}{|\mathcal{T}_1|!\ldots|\mathcal{T}_i|!|\mathcal{R}_1|!\ldots|\mathcal{R}_j|!}\;\:\:   
\end{equation}
When the number of types ($i+j$) is not fixed but grows with $n$, 
this (\ref{eq:multiFac}) is already superexponential in $n$.
For example, with $i+j=n$ and 
$|\mathcal{T}_1| = \ldots = |\mathcal{T}_i| = 
 |\mathcal{R}_1| = \ldots = |\mathcal{R}_j| = 1$, $C=n!$.
In any case, $C$ is at least exponential.
 
More to the point, to solve a noisy source identification problem, you then have to
recursively choose a subset of the possible configurations in order to construct a question.
That is $2^C$ possibilities.
And you need to choose a source $s_k$ to ask, 
making it $n\cdot 2^C$ possibilities, already.
Then, since sources can be random, it is also unclear
how fast the subproblems get smaller.
%In particular, when $\sum R$...  $j$ is close to $i$, e.g.\ $j=i-1$,

All this makes the total search space astronomically large. 
%For example
Illustratively, 
in practice,
taking roughly cubic time (in $C$) to make good choices, rather than say linear,
%can very well 
often 
results in shorter overall search times
(and much smaller final solutions),
because poor choices shrink the subproblems less.

\vspace{-0pt}
\section{A top-down solution to the puzzle}

\begin{solution}[Tim Roberts's solution]
\label{solution:nonRPrio}
Here is a solution to The Hardest Logic Puzzle Ever that is similar to Tim Roberts's
solution.\cite{rob01}
\vspace{-0.0em}
\end{solution}
\begin{proof}
We'll start by asking $\gamma_1$ the question $t_q(\gamma_3=\mathcal{R},\gamma_1)$.
Consider the possible cases:

\begin{description}[itemsep=0pt,topsep=1pt,leftmargin=1.1em]
%\vspace{-0.55em}
%\setlength{\itemsep}{0pt}
\item[Case \(t(\gamma_3=\mathcal{R},\gamma_1)\):] \hspace{0.0em}If $\gamma_1\neq\mathcal{R}$,
 then $\gamma_3=\mathcal{R}$ by theorem \ref{theorem:templateWorks}, and 
 $\gamma_2\neq\mathcal{R}$.
 If $\gamma_1=\mathcal{R}$,
 then $\gamma_2\neq\mathcal{R}$ again. Hence we now know that $\gamma_2\neq\mathcal{R}$ and 
 is safe to question.
 
 Next we'll ask $\gamma_2$ about $\gamma_1\neq\mathcal{R}$ (because there are $4$ 
 possibilities left, half of which have $\gamma_1=\mathcal{R}$).
 
 If \(t(\gamma_1\neq\mathcal{R},\gamma_2)\), then $\gamma_3=\mathcal{R}$,  
 and e.g.\ $t(\gamma_2=\mathcal{T},\gamma_2)$ determines which is which of $\gamma_1$ and 
 $\gamma_2$. 

 If \(\neg t(\gamma_1\neq\mathcal{R},\gamma_2)\), 
 then $\gamma_1=\mathcal{R}$, and 
 $t(\gamma_2=\mathcal{T},\gamma_2)$ determines which is which of $\gamma_2$ and $\gamma_3$.

\item[Case \(\neg t(\gamma_3=\mathcal{R},\gamma_1)\):] \hspace{0.0em}If 
$\gamma_1\neq\mathcal{R}$,
 then $\gamma_3\neq\mathcal{R}$ by theorem \ref{theorem:templateWorks}.
 If $\gamma_1=\mathcal{R}$,
 then $\gamma_3\neq\mathcal{R}$ again. Hence $\gamma_3\neq\mathcal{R}$ and is safe to question.
 
 Next we'll ask $\gamma_3$ about $\gamma_1\neq\mathcal{R}$.
%  (because there are $4$ possibilities left, half of which have $\gamma_1=\mathcal{R}$).

 If \(t(\gamma_1\neq\mathcal{R},\gamma_3)\), then $\gamma_2=\mathcal{R}$,  
 and e.g.\ $t(\gamma_3=\mathcal{T},\gamma_3)$ determines which is which of $\gamma_1$ and 
 $\gamma_3$. 

 If \(\neg t(\gamma_1\neq\mathcal{R},\gamma_3)\), 
 then $\gamma_1=\mathcal{R}$, and 
 $t(\gamma_3=\mathcal{T},\gamma_3)$ determines which is which of $\gamma_2$ and $\gamma_3$.
\end{description}
\vspace{-1.35em}
\end{proof}
%\vspace{-0.15em}

%\enlargethispage{1.10\baselineskip}

\section{A note on mathematical vs.\ computational thinking}
\label{sec:fixNote}
Donald Knuth has distinguished
%noted a distinction 
between 
mathematical and computational thinking.\cite{knuthInt}
%It ought to not be anything very profound --- 
%it's probably more about exposure and training.\cite{knuthInt}

%For example,
%An example of a particular type of
%One difference between computational reasoning and 
%a certain type of
%mathematical thinking is provided by
G\"{o}del's incompleteness theorems
provide a particular type of mathematical thinking.
Their proofs 
consist
%comprise
of
%are about 
straightforward computational reasoning,
except for the fixed-point theorem,
%(for any formula),
which 
%very much 
requires a certain mathematical thinking.
Similarly, the foundation of computability is straightforward and computational, 
except for its fixed-point result, Kleene's recursion theorem,
which is quite mathematical.\cite{fra13,yan03}
%,pel21,owi73,owi89,west10,ham10,

% Devlin(?)(no?) has a ``graphical'' proof of Kleene's recursion theorem,
% %(without any graphs),
% trying to 
% %gain some computational understanding of
% shed some light on
% it.\cite{dev} 
% Something similar works for G\"{o}del's fixpoint proof too.
% It's unclear though if \mbox{(non-)}Devlin's proof helps ---
% for starters it contains no graphs, only text, and it is not that short.
% Allthewhile the mathematical fixpoint proofs are just a few lines long.

% The Smntp(?) result from the foundation of computability is trivial computational bookkeeping. 
% But some mathematician back in the days apparently thought it prominent enough to
% warrant a name.
% Nowadays the result might be more confusing than clarifying.

% %Here,
% In this article,
% while the rest is mostly straightforward computational reasoning, 
% the meta-question template,
% def.\,\ref{def:tq},
% is more of the mathematical 
% %intelligence
% %type
% kind,
% perhaps.
%
The meta-question template in this paper,
def.\,\ref{def:tq},
is more of the mathematical 
%intelligence
%type
kind too, perhaps.
You can maybe even view the problem of finding a meta-question as
% kind of
finding a fixed-point, 
%in some way, 
in some 
%quirky 
setting.
% of sorts, in a strained and inverted and contrived way.
Cf.\ \cite{yan03}.
%,pel21
%,ham10
%
%
% Let d(g,q): d(R,(f(q))) = random{"Xi","-Xi" (0,1)},
% d(T,(f(q))) =  (f(q)) =  f(q),
% d(F,(f(q))) = -(f(q)) = -f(q).
% Find f, f', s.t. f'(d(g,q)) = q, for all q and non-randon g.
%
%
% let t(q) stand for t(q,\gamma), with \gamma \mathcal{T} or \mathcal{F}.
% Let \true be the application of t(q), or q. Then finding t is akin to 
% finding a fixed-point for \true.
% And for t to be a fixed-point for \true we want
% \true(\true(...(\true(t))...)) \bigleftrightarrow t.          <=> \true^0(t)=q???
% ^0 fixed-point might be okay with some sort of lambda style fixed-point approach??
% With t(q) = q(t). And q^n, n=0,1,2,... It's inverted and strained and artificial.
% Compared to how finding fixed-points are usually presented. Maybe not helpful??
% With the extra = \chi for the special circumstances here 
% (and =T for the puzzle where \chi is known and there is no R)
% cf. maybe \cite{pel21}
\vspace{0.8em}

%\appendix
\section{Infinite number of gods}
\label{sec:infGods}
The results in section \ref{sec:nGods} about when a puzzle is solvable 
hold also when the number of gods is infinite.

Let $\nu$ be the number of gods. 
Regard $\nu$ as an ordinal and
let the gods be
$\Gamma\!\coloneqq\bigcup_{\alpha<\nu}\{\gamma_\alpha\}$.

\subsection{Finding a non-random god}

For finding which type each god is, we'll define a function
$\bar{\varrho}$ that 
takes a well-ordered set of gods and
returns a non-random god, if there are more non-random than 
random gods.

Let $\bar{\varrho}(\{\gamma\})\coloneqq\gamma$.
Let $\bar{\varrho}(\{\gamma,\_\})\coloneqq\gamma$, 
with $\gamma$ least,
say.

For the successor case,
we'll ask the last god, $\gamma'$.
Like in lemma \ref{lemma:findingNonR}, we'll remove two gods, at least one of them a random god, if there is any random god at all.
Go through the gods, from least to greatest, 
until a random god $\gamma_\mathcal{\scriptscriptstyle R}$ is found,
i.e.\ until 
$t(\gamma_\mathcal{\scriptscriptstyle R}=\mathcal{R},\gamma')$,
or the search has reached the last god
in which case we'll set $\gamma_\mathcal{\scriptscriptstyle R}$ to the least god.
Then remove the last god, $\gamma'$;  
replace $\gamma_\mathcal{\scriptscriptstyle R}$ with the least god;
and update the well-ordering accordingly, by removing the last god,
and handling the $\gamma_\mathcal{\scriptscriptstyle R}$ replacement.

Then we can recursively apply $\bar{\varrho}$ to the new set of gods and
the updated well-ordering;
% that is now missing the last god, $\gamma'$, and
% $\gamma_\mathcal{\scriptscriptstyle R}$;
the result of this recursive application will be the result of the successor case too.
%0th to $\gamma_\mathcal{R}$, 1st to 2nd, ... \\
%because of well-orderedness/well-foundedness.\\

For the limit case,
the result of $\bar{\varrho}$ is the limit of $\bar{\varrho}$ on the smaller sets.
The limit exists if the non-random gods are more than the random gods.
% if the limit exists. Let it be some arbitrary but fixed and
% determined and defined god if the limit doesn't exist.\\
% The limit might not exist:\\
% E.g.\ consider a gods configuration given by:\\
% $(\mathcal{R}\mathcal{R}\mathcal{R}\ldots) + (\mathcal{T}\mathcal{T}\mathcal{T}\ldots) +
%  (\mathcal{T}\mathcal{T}\mathcal{T}\ldots) + \ldots$\\
% But eventually it will, and $\bar{\varrho}$ will be 

% -----------------------------
%
% what about the find non-R lemma, and its limit case? and limit+1?\\
% limit+1 is fine and works just like succ.\\
% E.g.\, what about\\
% $RRRRRRRRRR\!\leadsto\!\rightsquigarrow\!\!RRRTTTTTTTTTTT\!\ldots$?\\
% also possible for infinte number of Rs less than $\nu$.\\
% define function from gods to non-random god:\\
% succ: ask last; remove last and replace random with first.\\
% make second god the first, and so on.\\
%
% limit: function maps to \emph{limit} of function on less than induction case.\\
%
% and lemma condition still holds, i.e.\ that randoms are less
% than non-randoms (cardinal sense? ordinal sense?).\\
%
% always search from lowest to highest, until random god, or search has come to last god,\\
% because of well-orderedness/well-foundedness.\\
%
% %succ+1:
% -----------------------------

\begin{lemma}
\label{lemma:RsLTNonRsSolvableInf}
If a $\nu$ gods puzzle has strictly more non-random gods than random gods, 
then it is solvable.  
\end{lemma}
\vspace{-1.2em}
\begin{proof}
Assume more non-random gods than random gods. 
Then we can use $\bar{\varrho}$ to find a non-random god
$\gamma$. 
After that we can go through all the gods and ask $\gamma$ about their type.
%\vspace{-0.1em}
\end{proof}

\subsection{Finding an unsolvable puzzle}

Let $\mathscr{P}(\nu)$ be the set of all subsets of ordinals less than $\nu$.

\begin{theorem}
\label{theorem:nuGodsSolvability}
A $\nu$ gods puzzle is solvable if and only if the number of random gods is strictly less than
the number of non-random gods.
\end{theorem}
\begin{proof}%[Proof of theorem \ref{theorem:nGodsSolvability}]
The $\Leftarrow$ case is covered by lemma \ref{lemma:RsLTNonRsSolvableInf}.

% p. 8, l
%\enlargethispage{0.2\baselineskip}

For the $\Rightarrow$ case,
consider the easiest to solve of such problems:
Assume 
%that $n$ is even, 
that the random gods equal the non-random gods,
and that there are only truthful non-random gods.

$p_n$ (def.\,\ref{def:pnUns}) from section \ref{sec:nGods} becomes
\begin{defn}
\label{def:PnuUns}
$P_\nu \coloneqq $
 \begin{equation}
 \label{eq:nonSolvable2}
  \bigcup_{\underset{|\alpha|<\aleph_0}{\alpha\in\mathscr{P}(\nu),}}
   \{\,
   \bigvee
    \left\lbrace
     \begin{matrix}
      \bigwedge\limits_{\beta\in\alpha}
      \left\lbrace
       \begin{cases}
        \gamma_\beta\!=\!\mathcal{T},&\text{if } \beta \text{ is even}\\
        \gamma_\beta\!=\!\mathcal{R},&\text{if } \beta \text{ is odd}
       \end{cases}
      \right\rbrace\!\!,\\
      \bigwedge\limits_{\beta\in\alpha}
      \left\lbrace
       \begin{cases}
        \gamma_\beta\!=\!\mathcal{R},&\text{if } \beta \text{ is even}\\
        \gamma_\beta\!=\!\mathcal{T},&\text{if } \beta \text{ is odd}
       \end{cases}
      \right\rbrace%\,
     \end{matrix}
    \right\rbrace
   \}
 \end{equation}
\end{defn}
Let 
$P_{\mathcal{\scriptscriptstyle T}\mathcal{\scriptscriptstyle R}}$ and
$P_{\mathcal{\scriptscriptstyle R}\mathcal{\scriptscriptstyle T}}$ be the
conjunctions of $P_\nu$, similar to section \ref{sec:nGods}.
\begin{defn}
\label{def:PTR}
$P_{\mathcal{\scriptscriptstyle T}\mathcal{\scriptscriptstyle R}} \coloneqq $
 \begin{equation}
 \label{eq:PTR}
  \bigcup_{\underset{|\alpha|<\aleph_0}{\alpha\in\mathscr{P}(\nu),}}
   \{
      \bigwedge\limits_{\beta\in\alpha}
      \left\lbrace
       \begin{cases}
        \gamma_\beta\!=\!\mathcal{T},&\text{if } \beta \text{ is even}\\
        \gamma_\beta\!=\!\mathcal{R},&\text{if } \beta \text{ is odd}
       \end{cases}
      \right\rbrace
   \}
 \end{equation}
\end{defn}
\begin{defn}
\label{def:PRT}
$P_{\mathcal{\scriptscriptstyle R}\mathcal{\scriptscriptstyle T}} \coloneqq $
 \begin{equation}
 \label{eq:PRT}
  \bigcup_{\underset{|\alpha|<\aleph_0}{\alpha\in\mathscr{P}(\nu),}}
   \{
      \bigwedge\limits_{\beta\in\alpha}
      \left\lbrace
       \begin{cases}
        \gamma_\beta\!=\!\mathcal{R},&\text{if } \beta \text{ is even}\\
        \gamma_\beta\!=\!\mathcal{T},&\text{if } \beta \text{ is odd}
       \end{cases}
      \right\rbrace
   \}
 \end{equation}
\end{defn}
Let 
$P_{\mathcal{\scriptscriptstyle T}\mathcal{\scriptscriptstyle R}}^\nu$ and
$P_{\mathcal{\scriptscriptstyle R}\mathcal{\scriptscriptstyle T}}^\nu$ be the
full descriptions of the puzzle instances:
\begin{defn}
\label{def:PTRnu}
$P_{\mathcal{\scriptscriptstyle T}\mathcal{\scriptscriptstyle R}}^\nu \coloneqq $
 \begin{equation}
 \label{eq:PTRnu}
  \bigcup_{\beta<\nu}
%   \{
      \left\lbrace
       \begin{cases}
        \gamma_\beta\!=\!\mathcal{T},&\text{if } \beta \text{ is even}\\
        \gamma_\beta\!=\!\mathcal{R},&\text{if } \beta \text{ is odd}
       \end{cases}
      \right\rbrace
%   \}
 \end{equation}
\end{defn}
\begin{defn}
\label{def:PRTnu}
$P_{\mathcal{\scriptscriptstyle R}\mathcal{\scriptscriptstyle T}}^\nu \coloneqq $
 \begin{equation}
 \label{eq:PRTnu}
  \bigcup_{\beta<\nu}
%   \{
      \left\lbrace
       \begin{cases}
        \gamma_\beta\!=\!\mathcal{R},&\text{if } \beta \text{ is even}\\
        \gamma_\beta\!=\!\mathcal{T},&\text{if } \beta \text{ is odd}
       \end{cases}
      \right\rbrace
%   \}
 \end{equation}
\end{defn}
%
%
% p. 9, l
%\enlargethispage{0.7\baselineskip}
%
%
Assume that the gods are as described by 
$P_{\mathcal{\scriptscriptstyle T}\mathcal{\scriptscriptstyle R}}^\nu$ or
$P_{\mathcal{\scriptscriptstyle R}\mathcal{\scriptscriptstyle T}}^\nu$, and
that hence $P_\nu$ holds.

%Have the random gods ``happen''
%to force the puzzle there. To that end,
%%given already known conclusions $Q$,
If both $q$ and $\neg q$ are consistent with $P_\nu$,
%(i.e.\ not $p_n\!\rightarrow\!\neg q$ and not $p_n\!\rightarrow\!q$),
assume that the random gods always happen to give the incorrect answer
to a question about $q$.
If only one of $q$ and $\neg q$ is consistent with $P_\nu$,
assume that the random gods always happen to give the correct answer 
to a question about $q$.

% --------------------
%
% Skip? Since there can be many Rs.
%
% First we'll show that we are able to conclude that any sentence $\phi\in P_\nu$ 
% hold, from (finitely many) questions asked.
%
% Assume, without loss of generality, that $\phi \leftrightarrow$
% \begin{equation}
%   \bigvee\!
%    \left\lbrace
%     \begin{matrix}
%      \land\{\gamma_{\kappa_1}\!=\!\mathcal{T},\gamma_{\kappa_2}\!=\!\mathcal{R},
%             \ldots,\gamma_{\kappa_n}\!=\!\mathcal{T}\},\\
%      \land\{\gamma_{\kappa_1}\!=\!\mathcal{R},\gamma_{\kappa_2}\!=\!\mathcal{T},
%             \ldots,\gamma_{\kappa_n}\!=\!\mathcal{R}\}\;\\
%     \end{matrix}
%    \right\rbrace
% \end{equation}
%
%
% To see that $p_n$ will be concluded if non-trivial questions are asked,
% note that what we are able to conclude are
% $q \vee \gamma_1\!=\!\mathcal{R}$ from $t(q,\gamma_1)$, and
% $\neg q \vee \gamma_1\!=\!\mathcal{R}$ from $\neg t(q,\gamma_1)$. 
% Since $\gamma_1\!=\!\mathcal{R}$ is consistent with $p_n$, any disjunction with it is too.
% (Alternatively, 
% if $\gamma_1\!=\!\mathcal{T}$, then an answer that $q$ (or $\neg q$) holds
% implies that $q$ (or $\neg q$) must be consistent with $p_n$ since $p_n$ is in fact true,
% and $\gamma_1$ is non-random. 
% %Besides, $\gamma_1\!=\!\mathcal{R}$ too is consistent with $p_n$.
% And if $\gamma_1\!=\!\mathcal{R}$, then again $q$ (or $\neg q$) is consistent with $p_n$  
% because of the assumptions of how random gods answer.)
% %1) q,-q both consi
% %2) only one consi

We should be able to conclude that $P_\nu$ (eq.\,\eqref{eq:nonSolvable2}) holds,
if we reason in 
%e.g.\ 
set theory.
Because we can ask more gods than there are random gods about 
$P_\nu$.
Since everyone will answer that 
$P_\nu$ holds we can conclude that since 
we have asked at least one truthful god, $P_\nu$ must hold.

% If nothing else, eventually $\frac{n}{2}+1$ gods are asked about $p_n$ explicitly.
% Since all gods answer that $p_n$ holds, $p_n$ can be concluded since at least
% $1$ of $\frac{n}{2}+1$ gods must be non-random.

% Skip above?
%
% ------------------------------

Next we'll show that if 
%$q$ is non-trivial,
$P_\nu\!\nvdash q$ and 
$P_\nu\!\nvdash \neg q$,
then 
\(q \Leftrightarrow
  P_{\mathcal{\scriptscriptstyle T}\mathcal{\scriptscriptstyle R}}^\nu \)
or
\(q \Leftrightarrow
  P_{\mathcal{\scriptscriptstyle R}\mathcal{\scriptscriptstyle T}}^\nu \).

Assume 
$P_\nu\!\nvdash q$ and 
$P_\nu\!\nvdash \neg q$.

Then $q$ has a model $\mathcal{M}$ (i.e.\ an assignment of the gods)
where $q$, and $P_\nu$, are true. Suppose, without loss of generality, that it's 
$P_{\mathcal{\scriptscriptstyle T}\mathcal{\scriptscriptstyle R}}^\nu$
that's true in $\mathcal{M}$.
Then, since $P_{\mathcal{\scriptscriptstyle T}\mathcal{\scriptscriptstyle R}}^\nu$
completely determines a model, $q$ holds whenever 
$P_{\mathcal{\scriptscriptstyle T}\mathcal{\scriptscriptstyle R}}^\nu$ does.

Suppose not
$P_{\mathcal{\scriptscriptstyle T}\mathcal{\scriptscriptstyle R}}^\nu$.
Then 
$P_{\mathcal{\scriptscriptstyle R}\mathcal{\scriptscriptstyle T}}^\nu$.
Since 
$P_{\mathcal{\scriptscriptstyle R}\mathcal{\scriptscriptstyle T}}^\nu$
determines its models completely, if $q$ where to hold, then $q$ would
hold in all models to $P_\nu$, contradicting
% that
% $q$ is non-trivial.
$P_\nu\!\nvdash q$.
Hence $\neg q$ holds. That is, 
if not
$P_{\mathcal{\scriptscriptstyle T}\mathcal{\scriptscriptstyle R}}^\nu$,
then $\neg q$.

Assume that 
$P_{\mathcal{\scriptscriptstyle T}\mathcal{\scriptscriptstyle R}}^\nu$ or
$P_{\mathcal{\scriptscriptstyle R}\mathcal{\scriptscriptstyle T}}^\nu$ holds,
and that hence $P_\nu$ holds.
%Assume $P_\nu$ (eq.\,\eqref{eq:nonSolvable2}).
To see that once we know that
$P_{\mathcal{\scriptscriptstyle T}\mathcal{\scriptscriptstyle R}}^\nu$ or
$P_{\mathcal{\scriptscriptstyle R}\mathcal{\scriptscriptstyle T}}^\nu$ holds,
%To see that if $P_\nu$ is known, 
no more conclusions can be made, 
suppose that $\gamma_{\alpha}\!=\!\mathcal{T}$, with $\alpha$ even say.
Suppose $\gamma_\alpha$ is asked about $q$.
Suppose
% that $q$ is non-trivial.
$P_\nu\!\nvdash q$ and 
$P_\nu\!\nvdash \neg q$.
Then an answer that $q$ (or $\neg q$) holds 
(with implications that 
$P_{\mathcal{\scriptscriptstyle T}\mathcal{\scriptscriptstyle R}}^\nu$ holds)
is undone because we can only conclude that
$q \vee \gamma_{\alpha}\!=\!\mathcal{R}$ (or 
$\neg q \vee \gamma_{\alpha}\!=\!\mathcal{R}$),
which is equivalent to 
formulas in $P_\nu$.

% p. 8, r
%\enlargethispage{0.2\baselineskip}

Similarly, if instead $\gamma_\alpha=\mathcal{R}$, then we are able to conclude only
$q \vee \gamma_\alpha\!=\!\mathcal{R}$ (or $\neg q \vee \gamma_\alpha\!=\!\mathcal{R}$).
If 
%$q$ is non-trivial, 
%then
both $q$ and $\neg q$ are consistent with $P_\nu$,
the false answer that $q$ (or $\neg q$) holds
undoes the disjunct $\gamma_\alpha=\mathcal{R}$. 

If $q$ is already known, 
%trivial,
the concluded disjunction, 
$q \vee \gamma_1\!=\!\mathcal{R}$ (or $\neg q \vee \gamma_1\!=\!\mathcal{R}$),
will already be known too, and nothing new can be concluded again.
\end{proof}

\appendix

\section{A 4.1375 solver solution to the 0-3-2 problem}
\label{app:solverSolution}

Below is a solution from the solver$^{\ref{fn:prog}}$ of the 
$0$-$3$-$2$ problem with $0$ false gods, $3$ true gods, and
$2$ random gods.
%, with annotations.
%The text is slightly edited

The approach and outline are similar to that of our $4.15$ solution
in section \ref{sec:t3r2}.

The gods are named $g_0, g_1, \ldots, g_{n-1}$, for $n$ gods.
Conjunctions are written, for example, \texttt{RTTRT}, meaning
$\land\{g_0=\mathcal{R}, g_1=\mathcal{T}, g_2=\mathcal{T},
g_3=\mathcal{R}, g_4=\mathcal{T}\}$.
Disjunctions are written with each disjunct on a new line.
For example,
\vspace{-0.6em}
\begin{verbatim}
RT 
TR 
\end{verbatim}
\vspace{-0.6em}
means
$\vee\{\land\{g_0=R,g_1=T\},\land\{g_0=T,g_1=R\}\}$.

\texttt{P} is the probability $P$ that a disjunct is reached
in a particular way, as in \eqref{nOfQ} in section \ref{sec:avg},
and discussed in section \ref{sec:algo}. 
\texttt{Qs} is the number of questions used to detect a disjunct
in a particular way, as in \eqref{nOfQ}.

Each local disjunction is split in half and the upper half is used
as the next question (\texttt{q}). This split is marked by \texttt{-{}-{}-{}-{}-}.

There is 
%explicit and 
deliberate swapping of disjuncts between positive and
negative sides, to minimize the risk of asking random gods.

%./hardest -f 0 -t 3 -r 2 -b 0 -i 0 -v
\if 01
% new, v. 0.15.5:
\begin{verbatim}
|gods|=5 |fGods|=0 |tGods|=3 |rGods|=2 
possibilities: 10

asking g0:
g0=R possibilities: 4
g0=R <->
RTTTR
RTTRT
RTRTT
RRTTT

all possibilities <->
TTTRR
TTRTR
TTRRT
TRTTR
TRTRT
TRRTT
RTTTR
RTTRT
RTRTT
RRTTT

randoms at g1 for  q (not counting g0=R): 1
randoms at g2 for -q (not counting g0=R): 1

after swapping (not printing g0=R conjunctions):
TTRTR
TTRRT
TRRTT
-----
TRTTR
TTTRR
TRTRT

randoms at g1 for  q (not counting g0=R): 1
randoms at g2 for -q (not counting g0=R): 0

  asking g1, with 7 disjuncts left:
  g1=R possibilities: 2
  local premise <->
  TRRTT
  RRTTT
  RTRTT
  RTTRT
  TTRTR
  TTRRT
  RTTTR
  
  odd #: positive case has one more conjunction
  randoms at g3 for  q (counting g1=R): 0
  randoms at g4 for -q (counting g1=R): 0

  after swapping (not printing g1=R 
  conjunctions):
  RTRTT
  RTTTR
  TTRTR
  -----
  TTRRT
  RTTRT
  
  randoms at g3 for  q (counting g1=R): 0
  randoms at g4 for -q (counting g1=R): 0
\end{verbatim}
\vspace{-22pt}
%
% p. , l
%\enlargethispage{0.6\baselineskip}
%
\begin{verbatim}

    asking g3, with 5 disjuncts left:
    g3 is non-random, with log2(5)=2.321928.
    disjunct reached, P=1/2^0, Qs=4:
    TTRTR
    disjunct reached, P=1/2^1, Qs=4:
    RTRTT
    disjunct reached, P=1/2^1, Qs=4:
    RTTTR
    disjunct reached, P=1/2^2, Qs=5:
    RRTTT
    disjunct reached, P=1/2^1, Qs=5:
    TRRTT

    asking g4, with 4 disjuncts left:
    g4 is non-random, with log2(4)=2.000000.
    disjunct reached, P=1/2^0, Qs=4:
    TTRRT
    disjunct reached, P=1/2^1, Qs=4:
    TRRTT
    disjunct reached, P=1/2^1, Qs=4:
    RTTRT
    disjunct reached, P=1/2^2, Qs=4:
    RRTTT

  asking g2, with 7 disjuncts left:
  g2=R possibilities: 1
  local premise <->
  RTRTT
  TTTRR
  TRTRT
  RTTRT
  RRTTT
  RTTTR
  TRTTR
  
  randoms at g1 for  q (counting g2=R): 0
  randoms at g2 for -q (counting g2=R): 1

  after swapping (not printing g2=R 
  conjunctions):
  RTTTR
  RTTRT
  TTTRR
  -----
  RRTTT
  TRTRT
  TRTTR
  
  randoms at g1 for  q (counting g2=R): 0
  randoms at g2 for -q (counting g2=R): 1

    asking g1, with 4 disjuncts left:
    g1 is non-random, with log2(4)=2.000000.
    disjunct reached, P=1/2^0, Qs=4:
    TTTRR
    disjunct reached, P=1/2^1, Qs=4:
    RTTTR
    disjunct reached, P=1/2^1, Qs=4:
    RTTRT
    disjunct reached, P=1/2^2, Qs=4:
    RTRTT

    asking g2, with 4 disjuncts left:
    g2=R possibilities: 1
    local premise <->
    RTRTT
    RRTTT
    TRTRT
    TRTTR
    
    odd #: positive case has one more conjunction
    randoms at g3 for  q (counting g2=R): 0
    randoms at g4 for -q (counting g2=R): 0

    after swapping (not printing g2=R 
    conjunctions):
    RRTTT
    TRTTR
    -----
    TRTRT
    
    randoms at g3 for  q (counting g2=R): 0
    randoms at g4 for -q (counting g2=R): 0
\end{verbatim}
\vspace{-22pt}
%
% p. 10, r
%\enlargethispage{0.6\baselineskip}
%
\begin{verbatim}

      asking g3, with 3 disjuncts left:
      g3 is non-random, with log2(3)=1.584963.
      disjunct reached, P=1/2^0, Qs=4:
      TRTTR
      disjunct reached, P=1/2^1, Qs=5:
      RRTTT
      disjunct reached, P=1/2^3, Qs=5:
      RTRTT

      asking g4, with 2 disjuncts left:
      g4 is non-random, with log2(2)=1.000000.
      disjunct reached, P=1/2^0, Qs=4:
      TRTRT
      disjunct reached, P=1/2^3, Qs=4:
      RTRTT

Result, average number of questions asked to 
solve the problem: 4.137500
\end{verbatim}
\fi

%old:
\begin{verbatim}
|gods|=5 |fGods|=0 |tGods|=3 |rGods|=2 
possibilities: 10

asking g0:
g0=R possibilities: 4
g0=R <->
RTTTR
RTTRT
RTRTT
RRTTT

all possibilities <->
TTTRR
TTRTR
TTRRT
TRTTR
TRTRT
TRRTT
RTTTR
RTTRT
RTRTT
RRTTT

randoms at g1 for  q (not counting g0=R): 2
randoms at g2 for -q (not counting g0=R): 1

after swapping (not printing g0=R conjunctions):
TTRTR
TTRRT
TRRTT
-----
TTTRR
TRTRT
TRTTR

randoms at g1 for  q (not counting g0=R): 1
randoms at g2 for -q (not counting g0=R): 0

  asking g1, with 7 disjuncts left:
  g1=R possibilities: 2
  local premise <->
  TRRTT
  RRTTT
  RTRTT
  RTTRT
  TTRTR
  TTRRT
  RTTTR

  odd #: negative case has one more conjunction
  randoms at g4 for  q (counting g1=R): 0
  randoms at g3 for -q (counting g1=R): 1

  after swapping (not printing g1=R 
  conjunctions):
  TTRRT
  RTTRT
  -----
  TTRTR
  RTRTT
  RTTTR

  randoms at g4 for  q (counting g1=R): 0
  randoms at g3 for -q (counting g1=R): 0
\end{verbatim}
\vspace{-30pt}
%
% p. 10, l
%\enlargethispage{1.0\baselineskip}
%
\begin{verbatim}

    asking g4, with 4 disjuncts left:
    g4 is non-random, with log2(4)=2.000000.
    disjunct reached, P=1/2^0, Qs=4:
    TTRRT
    disjunct reached, P=1/2^1, Qs=4:
    TRRTT
    disjunct reached, P=1/2^1, Qs=4:
    RTTRT
    disjunct reached, P=1/2^2, Qs=4:
    RRTTT

    asking g3, with 5 disjuncts left:
    g3 is non-random, with log2(5)=2.321928.
    disjunct reached, P=1/2^0, Qs=4:
    TTRTR
    disjunct reached, P=1/2^1, Qs=4:
    TRRTT
    disjunct reached, P=1/2^1, Qs=4:
    RTRTT
    disjunct reached, P=1/2^2, Qs=5:
    RRTTT
    disjunct reached, P=1/2^1, Qs=5:
    RTTTR
\end{verbatim}
\begin{verbatim}

\end{verbatim}
\vspace{-40pt}
\begin{verbatim}
  asking g2, with 7 disjuncts left:
  g2=R possibilities: 1
  local premise <->
  RTRTT
  TRTRT
  TRTTR
  RTTTR
  RTTRT
  TTTRR
  RRTTT

  randoms at g2 for  q (counting g2=R): 1
  randoms at g1 for -q (counting g2=R): 1

  after swapping (not printing g2=R 
  conjunctions):
  TRTRT
  TRTTR
  RRTTT
  -----
  RTTRT
  TTTRR
  RTTTR

  randoms at g2 for  q (counting g2=R): 1
  randoms at g1 for -q (counting g2=R): 0
\end{verbatim}
%
% p. 10, r
%\enlargethispage{1.0\baselineskip}
%
\begin{verbatim}

\end{verbatim}
\vspace{-40pt}
\begin{verbatim}
    asking g2, with 4 disjuncts left:
    g2=R possibilities: 1
    local premise <->
    RTRTT
    TRTRT
    TRTTR
    RRTTT

    odd #: negative case has one more conjunction
    randoms at g4 for  q (counting g2=R): 0
    randoms at g3 for -q (counting g2=R): 0

    after swapping (not printing g2=R 
    conjunctions):
    TRTRT
    -----
    TRTTR
    RRTTT

    randoms at g4 for  q (counting g2=R): 0
    randoms at g3 for -q (counting g2=R): 0
\end{verbatim}
\vspace{-22pt}
% p. 10, r
%\enlargethispage{-1.2\baselineskip}
\begin{verbatim}

      asking g4, with 2 disjuncts left:
      g4 is non-random, with log2(2)=1.000000.
      disjunct reached, P=1/2^0, Qs=4:
      TRTRT
      disjunct reached, P=1/2^3, Qs=4:
      RTRTT
\end{verbatim}
\vspace{-14pt}
\begin{verbatim}
      asking g3, with 3 disjuncts left:
      g3 is non-random, with log2(3)=1.584963.
      disjunct reached, P=1/2^0, Qs=4:
      TRTTR
      disjunct reached, P=1/2^3, Qs=5:
      RTRTT
      disjunct reached, P=1/2^1, Qs=5:
      RRTTT
\end{verbatim}
\vspace{-14pt}
\begin{verbatim}
    asking g1, with 4 disjuncts left:
    g1 is non-random, with log2(4)=2.000000.
    disjunct reached, P=1/2^0, Qs=4:
    TTTRR
    disjunct reached, P=1/2^1, Qs=4:
    RTTRT
    disjunct reached, P=1/2^1, Qs=4:
    RTTTR
    disjunct reached, P=1/2^2, Qs=4:
    RTRTT

Result, average number of questions asked to 
solve the problem: 4.137500
\end{verbatim}

\subsection{Average number of questions used}

The table corresponding to \eqref{nOfQ} in section \ref{sec:avg}
for calculating the expected number of questions looks like this:

\begin{equation}
\label{nOfQSolver}
 \begin{matrix}
  P\,|Qs| & \gamma_0    & \gamma_1    & \gamma_2    & \gamma_3    & \gamma_4    \\
  \frac{1}{4}4,\frac{1}{4}5,\frac{1}{2}5
          & \mathcal{R} & \mathcal{R} & \mathcal{T} & \mathcal{T} & \mathcal{T} \\
  \frac{1}{2}5,\frac{1}{2}4
          & \mathcal{R} & \mathcal{T} & \mathcal{T} & \mathcal{T} & \mathcal{R} \\[1pt]
  \frac{1}{2}4,\frac{1}{2}4
          & \mathcal{R} & \mathcal{T} & \mathcal{T} & \mathcal{R} & \mathcal{T} \\
  4       & \mathcal{T} & \mathcal{T} & \mathcal{T} & \mathcal{R} & \mathcal{R} \\
  4       & \mathcal{T} & \mathcal{T} & \mathcal{R} & \mathcal{T} & \mathcal{R} \\
  4       & \mathcal{T} & \mathcal{T} & \mathcal{R} & \mathcal{R} & \mathcal{T} \\
  \frac{1}{2}4,\frac{1}{4}4,\mathit{\frac{1}{8}4},\mathit{\frac{1}{8}5}
          & \mathcal{R} & \mathcal{T} & \mathcal{R} & \mathcal{T} & \mathcal{T} \\
% 5       & \mathcal{R} & \mathcal{R} & \mathcal{T} & \mathcal{T} & \mathcal{T} \\
% 4       & \mathcal{T} & \mathcal{T} & \mathcal{R} & \mathcal{R} & \mathcal{T} \\

% 4       & \mathcal{R} & \mathcal{T} & \mathcal{R} & \mathcal{T} & \mathcal{T} \\
  4       & \mathcal{T} & \mathcal{R} & \mathcal{T} & \mathcal{T} & \mathcal{R} \\
  \frac{1}{2}4,\frac{1}{2}4 
          & \mathcal{T} & \mathcal{R} & \mathcal{R} & \mathcal{T} & \mathcal{T} \\
% 4       & \mathcal{R} & \mathcal{T} & \mathcal{T} & \mathcal{T} & \mathcal{R} \\

  4       & \mathcal{T} & \mathcal{R} & \mathcal{T} & \mathcal{R} & \mathcal{T} \\
% 4       & \mathcal{R} & \mathcal{T} & \mathcal{T} & \mathcal{R} & \mathcal{T} \\
% 5       & \mathcal{R} & \mathcal{R} & \mathcal{T} & \mathcal{T} & \mathcal{T} \\
% 5       & \mathcal{R} & \mathcal{T} & \mathcal{R} & \mathcal{T} & \mathcal{T} \\
% 4       & \mathcal{T} & \mathcal{R} & \mathcal{R} & \mathcal{T} & \mathcal{T} \\
 \end{matrix}
%\vspace{-0.15em}
\end{equation}
Thus, the average number of questions used to find a solution is
\( ( \frac{1}{4}4 + \frac{1}{4}5 + \frac{1}{2}5 + 
     \frac{1}{2}5 + \frac{1}{2}4 +
     \frac{1}{2}4 + \frac{1}{2}4 +
     5\cdot 4 +
     \frac{1}{2}4 + \frac{1}{4}4 + \frac{1}{8}4 + \frac{1}{8}5 +
     \frac{1}{2}4 + \frac{1}{2}4 ) \;/\; 10 = 4.1375 \).

Comparing the two tables, \eqref{nOfQ} and \eqref{nOfQSolver}, the solver traded
$\frac{1}{4}5$ for $\frac{1}{8}4 + \frac{1}{8}5$, 
hiding a discovery using $5$ questions deeper. 
The gain, 
$\frac{1}{4}5 = 1.25$ minus $\frac{1}{8}4 + \frac{1}{8}5 = 1.125$ is
$0.125$, which divided by $10$ disjuncts gives $0.0125$, which is the improvement,
$4.15-4.1375$, that the solver found.

% log size of hash table: 4
% number of hash table collisions: 3
% ht-size: 16
% more hash table stats, collisions and chain lengths:
% 0: 9
% 1: 4
% 2: 3
% load: 0.625
% number of entries: 10

\section{Solver solution for one random god}
\label{app:oneR}
Here is an example of the solutions that the solver finds on problems 
%with one random god.
where it is possible to get the second and third asked gods, $g_1$ and $g_2$,
to have no chance of being random.

The example also illustrates a fast, probabilistic way to estimate the expected number of questions 
at the end, 
dynamically during the search process.
%before the whole solution is found. 
Recall from section \ref{sec:algo} that 
the average number of questions used to conclude a disjunct is
$1/2^{r_1} q_1 + \ldots + 1/2^{r_k} q_k$, where $r_i$ is 
the number of random gods questioned on a path to a conclusion that this disjunct holds,
$q_i$ is the number of questions used for that particular conclusion,
and $k$ is the number of ways that the disjunct can be detected.
Given an estimate $e = (q_{e_1} + \ldots + q_{e_n}) / n$ of the expected number of questions,
and a newly found path to a disjunct that questioned $r_i$ random gods, and used $q_i$
questions, we will add $q_i$ to $e$, getting $e' = (q_{e_1} + \ldots + q_{e_n} + q_i ) / (n+1)$,
with probability $1/2^{r_i}$.
%You can probabalistically add
\begin{verbatim}
|gods|=5 |fGods|=1 |tGods|=3 |rGods|=1 
possibilities: 20

asking g0:
g0=R possibilities: 4
g0=R <->
RFTTT
RTFTT
RTTFT
RTTTF

all possibilities <->
FTTTR
FTTRT
FTRTT
FRTTT
TFTTR
TFTRT
TFRTT
TTFTR
TTFRT
TTTFR
TTTRF
TTRFT
TTRTF
TRFTT
TRTFT
TRTTF
RFTTT
RTFTT
RTTFT
RTTTF
\end{verbatim}
\vspace{-22pt}
\begin{verbatim}

randoms at g1 for  q (not counting g0=R): 4
randoms at g2 for -q (not counting g0=R): 3
\end{verbatim}
\vspace{-22pt}
\begin{verbatim}

after swapping (not printing g0=R conjunctions):
TFRTT
FTRTT
FTTTR
TTFTR
TTFRT
TTRFT
TTRTF
TTTFR
-----
TRFTT
FRTTT
TRTFT
TRTTF
TTTRF
TFTRT
TFTTR
FTTRT
\end{verbatim}
\vspace{-22pt}
\begin{verbatim}

randoms at g1 for  q (not counting g0=R): 0
randoms at g2 for -q (not counting g0=R): 0
\end{verbatim}
\vspace{-22pt}
\begin{verbatim}

  asking g1, with 12 disjuncts left:
  g1 is non-random, with log2(12)=3.584963.
  disjunct reached, P=1/2^0, Qs=4:
  TFRTT
  disjunct reached, P=1/2^0, Qs=4:
  FTRTT
  disjunct reached, P=1/2^0, Qs=4:
  FTTTR
  disjunct reached, P=1/2^0, Qs=4:
  TTFTR
  disjunct reached, P=1/2^1, Qs=5:
  RTFTT
  disjunct reached, P=1/2^1, Qs=5:
  RFTTT
  disjunct reached, P=1/2^1, Qs=5:
  RTTTF
  disjunct reached, P=1/2^1, Qs=5:
  RTTFT
  disjunct reached, P=1/2^0, Qs=5:
  TTFRT
  disjunct reached, P=1/2^0, Qs=5:
  TTRFT
  disjunct reached, P=1/2^0, Qs=5:
  TTRTF
  disjunct reached, P=1/2^0, Qs=5:
  TTTFR
\end{verbatim}
\vspace{-14pt}
\begin{verbatim}
state is promising: estimated result: 4.600000
\end{verbatim}
% p. 12, l
%\enlargethispage{0.7\baselineskip}
\vspace{-14pt}
\begin{verbatim}
  asking g2, with 12 disjuncts left:
  g2 is non-random, with log2(12)=3.584963.
  disjunct reached, P=1/2^0, Qs=4:
  TRFTT
  disjunct reached, P=1/2^0, Qs=4:
  FRTTT
  disjunct reached, P=1/2^0, Qs=4:
  TRTFT
  disjunct reached, P=1/2^0, Qs=4:
  TRTTF
  disjunct reached, P=1/2^0, Qs=5:
  TTTRF
  disjunct reached, P=1/2^0, Qs=5:
  TFTRT
  disjunct reached, P=1/2^0, Qs=5:
  TFTTR
  disjunct reached, P=1/2^0, Qs=5:
  FTTRT
  disjunct reached, P=1/2^1, Qs=5:
  RFTTT
  disjunct reached, P=1/2^1, Qs=5:
  RTTFT
  disjunct reached, P=1/2^1, Qs=5:
  RTFTT
  disjunct reached, P=1/2^1, Qs=5:
  RTTTF

Result: average number of questions asked to 
solve the problem: 4.600000
\end{verbatim}
\vspace{-0.5em}
% p. 12, r
%\enlargethispage{0.7\baselineskip}

% Widen bibliography:
%\usepackage[a4paper,top=1.12cm,bottom=1.91cm,left=1.48cm,right=1.48cm]{geometry}
%\newgeometry{left=1.38cm,right=1.38cm}
%\newgeometry{top=1.12cm,bottom=1.91cm,left=1.38cm,right=1.38cm}
%NB: Change below too, if changed!!
%\newcommand{\refwidadj}{0.02cm}
%\newcommand{\refwidadj}{0.08cm}
%\begin{adjustwidth}{-0.02cm}{-0.02cm}
%\begin{adjustwidth}{-\refwidadj}{-\refwidadj}

%\nocite{*}
%\nocite{gend17}
%\bibliographystyle{plain}
%\bibliographystyle{unsrt}
%\bibliographystyle{unsrtnat}
%\renewcommand{\refname}{\hspace{\refwidadj}Bibliography}
% For a small reference section:
%{\small
%{\footnotesize
\bibliography{how_to_solve_the_hardest_logic_puzzle_ever}
%}

%\end{adjustwidth}

\end{document}